\documentclass[acmsmall]{ec19acm}
\AtBeginDocument{%
	\providecommand\BibTeX{{%
			\normalfont B\kern-0.5em{\scshape i\kern-0.25em b}\kern-0.8em\TeX}}}

\usepackage{acm-ec-19}

\usepackage[ruled]{algorithm2e} 

\SetAlFnt{\small}
\SetAlCapFnt{\small}
\SetAlCapNameFnt{\small}
\SetAlCapHSkip{0pt}
\IncMargin{-\parindent}

\usepackage{macros}
\usepackage[utf8]{inputenc}
\usepackage{microtype}
\usepackage{graphicx}
\usepackage{subfigure}
\usepackage{booktabs} 
\usepackage{amsmath}
\usepackage{amsfonts}
\usepackage{amsthm}
\usepackage{hyperref}

\newcommand*{\GtrSim}{\smallrel\gtrsim}
\makeatletter
\newcommand*{\smallrel}[2][.85]{%
  \mathrel{\mathpalette{\smallrel@{#1}}{#2}}%
}
\newcommand*{\smallrel@}[3]{%
  \sbox0{$#2\vcenter{}$}%
  \dimen@=\ht0 %
  \raise\dimen@\hbox{%
    \scalebox{#1}{%
      \raise-\dimen@\hbox{$#2#3\m@th$}%
    }%
  }%
}
\makeatother

\copyrightyear{2019}
\acmYear{2019}
\setcopyright{rightsretained}
\acmConference[EC '19] {The 20th ACM Conference on Economics and Computation}{June 24--28, 2019}{Phoenix, AZ, USA}
\acmBooktitle{The 20th ACM conference on Economics and Computation (EC '19), June 24--28, 2019, Phoenix, AZ, USA}
\acmPrice{}
\acmDOI{10.1145/3328526.3329656}
\acmISBN{978-1-4503-6792-9/19/06}

\title{Fundamental Limits of Testing the Independence of Irrelevant Alternatives in Discrete Choice}
\titlenote{An extended abstract for an early version of this work appeared in the {\it Proceedings of the 20th ACM Conference on Economics and Computation (EC ’19)}, June 24--28, 2019, Phoenix, AZ. We thank Stephen Ragain, Amin Saberi, Kiran Shiragur, and Steve Yadlowsky for many helpful discussions and suggestions. We specifically thank Arun Jambulapati for help extending the short cycle decomposition existence results to hypergraph incidence graphs (Lemma \ref{lemma:bipartite_cycle_decomposition}). All errors are our own. This work has been supported in part by an ARO Young Investigator Award and a gift from the Koret Foundation. AS was also supported in part by an NSF Graduate Research Fellowship.
The authors can be reached at \{aseshadr, jugander\}$@$stanford.edu.}
\author{Arjun Seshadri}
\email{aseshadr@stanford.edu}
\orcid{0000-0003-1827-9831}
\affiliation{%
	\institution{Stanford University}
	\streetaddress{450 Serra Mall}
	\city{Stanford}
	\state{California}
	\postcode{94305}
	\country{USA}
}

\author{Johan Ugander}
\email{jugander@stanford.edu}
\affiliation{%
	\institution{Stanford University}
	\streetaddress{450 Serra Mall}
	\city{Stanford}
	\state{California}
	\postcode{94305}
	\country{USA}}

\begin{document}

	\newcommand{\xhdr}[1]{\paragraph*{\bf #1}}

	\newcommand\allsubsets{\mathcal{C}}
	\newcommand\uniqueD{\mathcal{C}_{\mathcal{D}}}
	\newcommand\Var{\text{Var}}
	\newcommand\Cov{\text{Cov}}
	\newcommand\nper{{n_{\text{per}}}}
	\newcommand\eff{\textbf{f}}
	\newcommand\exx{\mathcal{X}}
	\newcommand{\tv}[1]{||#1||_{\text{TV}}}
	\newtheorem{defn}{Definition}
	\newtheorem{thm}{Theorem}
	\newtheorem{cor}{Corollary}
	\newtheorem{ass}{Assumption}
	\newtheorem{question}{Question}
	
	\newcommand\Pall{\mathcal{P}_\mathcal{C}}
	\newcommand\Piia{\mathcal{P}^{\text{IIA}}_\mathcal{C}}
	\newcommand\Mdelta{\mathcal{M}_{\delta,\mathcal{C}}}
	\newcommand\Bset{\mathcal{B}_\mathcal{C}}
	\newcommand\qbe{q_{b,\epsilon}}
	\newcommand\qbpe{q_{b',\epsilon}}

	\newcommand{\jucom}[1]{\textcolor{red}{[#1]}}
	\newcommand{\ju}[1]{\textcolor{red}{#1}}
	\newcommand{\ascom}[1]{\textcolor{blue}{[#1]}}
	\newcommand{\kk}[1]{{\color{red}#1}}

	\begin{abstract}
		The Multinomial Logit (MNL) model and the axiom it satisfies, the Independence of Irrelevant Alternatives (IIA), are together the most widely used tools of discrete choice. The MNL model serves as the workhorse model for a variety of fields, but is also widely criticized, with a large body of experimental literature claiming to document real-world settings where IIA fails to hold. Statistical tests of IIA as a modelling assumption have been the subject of many practical tests focusing on specific deviations from IIA over the past several decades, but the formal size properties of hypothesis testing IIA are still not well understood. In this work we replace some of the ambiguity in this literature with rigorous pessimism, demonstrating that any general test for IIA with low worst-case error would require a number of samples exponential in the number of alternatives of the choice problem. A major benefit of our analysis over previous work is that it lies entirely in the finite-sample domain, a feature crucial to understanding the behavior of tests in the common data-poor settings of discrete choice.
		Our lower bounds are structure-dependent, and as a potential cause for optimism, we find that if one restricts the test of IIA to violations that can occur in a specific collection of choice sets (e.g., pairs), one obtains structure-dependent lower bounds that are much less pessimistic. Our analysis of this testing problem is unorthodox in being highly combinatorial, counting Eulerian orientations of cycle decompositions of a particular bipartite graph constructed from a data set of choices. By identifying fundamental relationships between the comparison structure of a given testing problem and its sample efficiency, we hope these relationships will help lay the groundwork for a rigorous rethinking of the IIA testing problem as well
		as other testing problems in discrete choice.
	\end{abstract}
	\maketitle
	
	\section{Introduction}
	A common goal for a wide variety of social-scientific pursuits is to infer preferences from the choices individuals make given a limited, discrete set of alternatives. Towards this goal, \textit{probabilistic discrete choice modelling} 
	is broadly applied in the study of
	transportation \cite{domencich1975urban}, 
	industrial organization \cite{berry1995automobile},
	marketing \cite{allenby1998marketing}, 
	network formation \cite{overgoor2019choosing}, 
	recommender systems \cite{resnick1997recommender}, 
	search engine ranking \cite{schapire1998learning}, 
	and other judgement or decision problems that are discrete in nature. 
 	
	The study of discrete choice focuses on the following problem setting. A decision maker, either an individual or a representative of a population, chooses an alternative $x$ from a choice set $C$ of two or more alternatives. The set of alternatives is a subset of a broader universe $\mathcal{X}$ of $n$ alternatives.  The choice set captures the restrictions faced by the chooser at the time the decision is made, either by design or due to constraints. The goal in modelling discrete choice is then to determine how choices are made by a decision maker (or a group of decision makers) across a specific collection $\mathcal{C}$ of unique subsets of $\mathcal{X}$. The collection $\mathcal{C}$ may be the collection of all subsets of $\mathcal{X}$, may be the collection of all pairs or of all triplets, or may be some empirically determined collection, and is generally treated as exogenous to the problem. The goal of probabilistic modelling is to model with what \textit{probability} each item $x$ will be chosen from each set $C \in \mathcal{C}$. The corresponding challenge of inference is to recover a model given a dataset of subsets and choices from those subsets (and the identity of the decision makers, if heterogeneous).
	
	Perhaps the most prevalent approach to modelling discrete choice is to assign a latent utility to all $n$ items of the universe, and then model decision making as a noisy maximization of the item utilities in a given choice set. The conditional Multinomial Logit (MNL) model \cite{mcfadden1977application}, also known as the Bradley-Terry-Luce model \cite{bradley1952rank,luce1959individual}, is the best known among these Random Utility Models \cite{falmagne1978representation} that follow this utility-based approach. Indeed, the MNL model is the foundational workhorse of probabilistic discrete choice modelling for settings both small and large in scale. 
	
	The MNL model is widely described as modelling ``rational'' discrete choice, a position that can be motivated in several ways. 
	A common motivation is a willingness to assume the axiom of the Independence of Irrelevant Alternatives (IIA) \cite{luce1959possible}, also known as Luce's choice axiom. A motivating consequence of IIA is that under IIA, adding an item to a choice set 
	does not change the relative probabilities of choosing the existing items in the original set. 
	Alternative motivations for the MNL model include a willingness to assume a random utility model along with either independent Gumbel noise, which characterizes the MNL model~\cite{manski1977structure} or by assuming an independent random utility model that satisfies the axiom of 
	uniform expansion~\cite{yellott1977relationship}. 
	Regardless of the motivation, the MNL model is unique among all probabilistic models of discrete choice in satisfying IIA, so IIA can be thought of as a complete characterization of the model within the space of all discrete choice models.
	
	The above idea that the MNL model is typically adopted as a result of careful deliberations about underlying assumptions is very generous towards practitioners: more often, the MNL model is adopted simply because it is easy to work with. It is easy to learn from data (the log-likelihood is convex, and the number of parameters scale linearly with $n$), it is easy to employ in prediction tasks (predicted choice probabilities have a simple functional form), and it is easy to interpret (the parameters of the model are interpretable as relative utilities of the items in the choice problem).
	
	Given that IIA fully characterizes the MNL model, the widespread popularity of the model has directed a great deal of attention to varied evaluations of whether IIA holds in a given context or dataset. A large body of experimental evidence claims to document significant deviations from IIA across many domains. Examples of specific deviations that have received considerable attention include the compromise effect \cite{simonson1989choice}, asymmetric dominance \cite{huber1982adding}, and the similarity effect \cite{tversky1972elimination}. As a result of these observations, which at least date back to Debreu's book review~\cite{debreu1960review} of Luce's original monograph, researchers have taken to developing models of discrete choice that eschew IIA; examples include the Elimination By Aspects (EBA) model \cite{tversky1972elimination},
	nested models \cite{mcfadden1978modeling}, mixture models \cite{mcfadden1975aggregate,mcfadden2000mixed}, Markov chain models \cite{blanchet2016markov, ragain2016pairwise}, and contextual random utility models \cite{simonson1989choice,batsell1985new,chen2016modeling,seshadri2019raw}.
	
	The above evaluations and modelling efforts have in no way discouraged practitioners from using the MNL model, even at modest item sizes 
	\cite{lurkin2017accounting,raval2019machine}.
	One reason is that the models that eschew IIA have almost always been more complex than the MNL model, either in the number of parameters (sometimes scaling combinatorially in $n$ \cite{park2013theoretical}) or in inferential difficulty \cite{keane1992note, chierichetti2018learning, benson2016relevance}, and their use thus requires specific justification. Meanwhile, most empirical efforts to reject IIA do so in settings of, e.g., three items and two subsets (a pair and the full triplet) \cite{simonson1989choice, huber1982adding, tversky1972elimination}. These demonstrations are highly illustrative, but not cleanly transferable to motivating problems with many more alternatives\footnote{While IIA violations in one setting could be an indicator of violations in another, different deviations studied often conflict with each other in direction and magnitude, making predictions difficult over a large number of unique subsets.}. This is to say, the small-scale demonstrations fall short of formally {\it testing} whether IIA holds for a general discrete choice problem. The demonstrations lack the objectivity, transparency, and interpretability that (properly used) testing procedures provide in practice \cite{johari2015always}. While formal tests for IIA do exist, such as the early proposals of McFadden, Tye, and Train \cite{mcfadden1977application} and of Hausman and McFadden ~\cite{hausman1984specification}, these tests have been problematic in practice, as discussed further in the prior work discussion within this introduction.
	When IIA cannot be rejected with confidence for a given problem, it is natural to default to the MNL model. Should the deviations from IIA that have been documented in small problem instances indeed hold for larger instances, these deviations are being largely ignored by many of today's practitioners. 

	In this work we ask: given a dataset of choices from choice sets within a collection $\mathcal{C}$, how efficiently can IIA be tested in a general $n$ alternative setting? The main difficulty with such a test lies in the combinatorial nature of discrete choice problems. Whereas it is well known that the MNL model can be inferred with samples scaling only in $n$ under IIA  \cite{negahban2012iterative, shah2016estimation, agarwal2018accelerated}, a test of IIA is inherently a matter of model \textit{misspecification}. For IIA to not hold, a sufficient condition is a departure from otherwise IIA-consistent choices by the choice probability of a single item on a single subset within the collection $\mathcal{C}$. A collection $\mathcal{C}$ with many subsets, notably subsets of large size, could then permit many distinct ways for departing from IIA and render all tests of IIA sample-intensive. This challenges follows what is sometimes called the Anna Karenina principle\footnote{The opening line of Tolstoy's {\it Anna Karenina} being: ``Happy families are all alike; every unhappy family is unhappy in its own way.'' \cite{Tolstoy}} of high-dimensional hypothesis testing: all nulls are alike, but deviations from the null all deviate in their own way. In our context, the principle translates to say that while there are only a few ways to be ``rational,'' there are a many unique ways that people can be ``irrational.''

	As part of our work, we do not focus on any particular test of IIA, but rather, we interrogate the sample complexity inherent to the IIA testing problem by bounding from below the worst case error of any IIA testing procedure. 
	Because we study the worst case error given a finite number of samples, our results are necessarily dependent on every set in $\mathcal{C}$ being allocated $\textit{some}$ of the samples. Were the structure of $\mathcal{C}$ ignored in this respect, with some sets receiving no samples, the worst case error of a test is always that of a random guess, since the "worst case" IIA deviation can be thought of as arising only in unsampled sets. 
	
	Our main contribution is to show that the sample complexity of testing IIA grows at least with the \textit{square root} of the sum of all subset cardinalities in the collection $\mathcal{C}$. This scaling yields the pessimistic outcome that testing for IIA over a collection with many and large subsets can quickly require samples exponential in the number of items $n$. Our bounds also yield room for optimism. In many practical settings, $\mathcal{C}$ only consists of specific pairs and triplets, not, e.g., all subsets of size seventeen. For these settings, where the sum of all subset cardinalities in the collection $\mathcal{C}$ scales mildly with $n$, our bounds suggest the possible existence of surprisingly efficient procedures for testing IIA. This efficiency stands in contrast to the classical chi-squared approaches that require an accurate estimate of the distributions over the subsets in $\mathcal{C}$ and hence samples \textit{linear} in the sum of subset cardinalities \cite{read2012goodness}. 
	
	Although traditional theoretical analyses show that $\chi^2$ tests for discrete distributions are rate-optimal in the minimax sense, these analyses (see \cite{lehmann2005testing}, pg.~593) treat the dimension as fixed and constant compared to a growing number of samples, and their resulting rates conceal the impact of problem size in relation to data. When dimensions are taken into account, estimation-based classical $\chi^2$ tests are no longer generally optimal, and tests are often feasible with far fewer samples than are required for estimation. This is shown in recent finite-sample analyses for a variety of basic testing problems over discrete distributions, including testing identity \cite{valiant2017automatic}, testing amongst monotone distributions \cite{wei2016sharp}, testing independence and many others properties \cite{daskalakis2018distribution}. The possible existence of efficient procedures for testing IIA for practical $\mathcal{C}$ connects our work to this broader project in  hypothesis testing properties of discrete distributions.

	\xhdr{Prior work on hypothesis testing IIA}
	The problem of testing IIA was first posed in Luce's influential book \cite{luce1959individual} introducing the MNL model. There, Luce considers a rudimentary test of whether the probability ratio of choosing two items from a two-item and three-item choice set are the same, as they should be under IIA. He performs an approximate calculation showing that estimating the two ratios with low variance in order to make an accurate claim about IIA seems to demand thousands of samples. 
	The first formal hypothesis tests for IIA were proposed in 1977 by McFadden, Tye, and Train \cite{mcfadden1977application}. These tests were later found to be asymptotically biased \cite{small1985multinomial} and inspired several alternatives. 
	
	All of these alternative tests---by Hausman-McFadden~\cite{hausman1984specification}, Small-Hsiao~\cite{small1985multinomial}, and Horowitz~\cite{horowitz1981identification}---rely on the same mechanism to detect violations: they construct tests for differences in parameter estimates between estimation done in the presence or absence of an item in its associated choice sets. McFadden \cite{mcfadden1987regression} has also proposed testing joint specifications in a regression framework. All these tests have been discovered to be problematic; Fry and Harris~\cite{fry1996monte,fry1998testing} first used simulations to show that the known tests have poor size properties, and that practical datasets were far too small to operate in their assumed asymptotic regime. Cheng and Long~\cite{cheng2007testing} later confirmed this finding through more rigorous simulation, but moreover demonstrated that data sampled from IIA models often reject the null hypothesis of this test even with extraordinarily large data sizes, and data sampled from other models often fail to reject. On the other hand, model based tests---tests that compare the MNL model to a more general model\footnote{Typical examples include the Generalized Extreme Value Model, Nested Logit, and the Mixed Multinomial Logit.} that nests the MNL model---are often unused because the models that eschew IIA suffer from computational intractability, or issues with identifiability \cite{cheng2007testing}. These results have since influenced folklore~\cite{long2005regression} which advises researchers against using existing tests for IIA. While we do not propose any new tests in this work---only lower bounds on the sample complexity of the best possible tests---we hope that our lower bound results (and proof techniques) can open the door to new constructive tests that test IIA rigorously, if not also efficiently. 
	
	\xhdr{Prior graphical work on discrete choice}
	Our work's graphical treatment of discrete choice models is not unique in the literature. Namely, we are not the first to use graph properties to provide finite sample minimax lower bounds. In \cite{negahban2012iterative}, random walks on pairwise comparison graphs form the basis for interpreting and analyzing the paper's main RankCentrality algorithm for inferring BTL model parameters. In \cite{shah2016estimation}, eigenvalues of a comparison graph's Laplacian are used to prove the minimax optimality of the maximum likelihood estimator for the MNL model. The same problem is approached in \cite{hajek2014minimax}, however with a greater emphasis on spectral gaps and graph degree distributions. Finally, Markov chains constructed from comparison graphs form the basis of multiple estimation algorithms, including \cite{maystre2015fast} and \cite{agarwal2018accelerated}. Every one of these works focuses on the estimation of MNL models, as opposed to testing their validity, the focus of our work. Novel about our work is the study of cycles in our comparison incidence graphs as the property crucial to departures from IIA. This property should not be confused with the cyclic monotonicity property of the MNL model used in \cite{shi2018estimating} to construct a consistent estimator for the model, as cyclic monotonicity is a convex-analytic property that holds true for MNL regardless of the comparison graph.
	
	\xhdr{The present work}
	
	We provide the first formal lower bounds on the complexity of testing IIA. To do so, we leverage  recent progress in the finite sample analysis of the complexity of testing discrete distributions over the last decade \cite{paninski2008coincidence,valiant2017automatic,wei2016sharp,acharya2015optimal} across the statistics, information theory, and theoretical computer science literatures. The first step in our approach is a testing relaxation that studies a statistically and analytically simpler test than the IIA test, and requires a careful construction stemming from a desire to keep the sufficient statistics of the MNL model unchanged. The relaxed test is then connected to directing cycles of a particular {\it comparison incidence graph} $G_\mathcal{C}$ based on the collection $\mathcal{C}$, special properties of which always ensure a cycle decomposition. It is then shown, through a mixture of statistical and combinatorial analysis, that the cycles help guarantee (through their size properties) many orthogonal departures from IIA, furnishing a high lower bound for the IIA testing problem. Notable in our analysis is the key role of the structure of the collection of unique comparison sets $\mathcal{C}$ where the IIA property is being tested. 
	
	Our main result is then a minimax lower bound on testing IIA that, informally, lower bounds the error of the best possible test on the worst case problem instance for a given comparison setting and a given number of samples. The main consequence of our lower bound is that the worst-case sample complexity of the IIA testing problem is lower-bounded by a quantity proportional to $\sqrt{d}$, where $d$ is the sum of the sizes of the subsets in $\mathcal{C}$. In the case where a researcher considers all settings where IIA could be violated, that is, over all possible subsets of a universe set of $n$ items, size 2 or greater, $d$ is exponential in $n$. Thus, testing for IIA, in its most general form, has a worst case sample complexity that is at least exponential in the number of items. 
	
	A secondary consequence of our lower bound is one of optimism (or at least a lack of pessimism) when $d$ is small. Although the discrete choice literature treats IIA as one general property capturing the relationships within a set of items across all subsets, we develop our analysis in a way that makes it possible to discuss restricted notions of IIA, where we may only test IIA as it applied to a given collection of choice sets. As discussed earlier, if a discrete choice task only involves choices from, e.g., sets of size two and three, we are not really concerned with violations of IIA that may involve specific pathological sets of size seventeen. In developing minimax lower bounds that are collection-specific, we can ask what limitations may exist (what the sample complexity is) for testing these specific types of violations of IIA. In short, this framework lets us investigate limits on ``problem-relevant irrationalities'', setting aside the overwhelming number of ways a person can be irrational without relevance for a given problem. For the specific case of pairwise comparisons, the lower bound furnished by our analysis is rather mild, scaling linearly in the number of items. We also study a particular cyclical comparison structure we dub the single big cycle, which results in lower bounds that are entirely ``dimension-free''. 
	Our lower bound can therefore be seen as optimistic in the cases of certain comparison structures: rationality is much easier to test if you restrict the number of irrationalities it is up against. 
	
	We organize the paper as follows. In Section~\ref{sec:preldef}, we introduce the discrete choice problem setting by defining primitives that jointly capture the statistical and combinatorial aspects of the problem. Along the way, we formally describe the Independence of Irrelevant Alternatives (IIA), and characterize discrete choice settings that violate it. In Section~\ref{sec:lowerbounds}, we introduce the quantity central to the fundamental limits of testing IIA---the minimax risk---and state the paper's main result. In Sections 4 and 5, we develop the foundational lemmas that underlie this main result, constructing first a test that is statistically simpler than a test for IIA, and then proving a lower bound on the difficulty of the simpler test. We combine these lemmas in Section 6 to prove the paper's main result. Finally, in Section 7, we explore consequences of the result, taking advantage of the structure-dependent aspects to identify testing scenarios where the lower bounds furnish particularly pessimistic and optimistic perspectives.
	\section{Preliminary definitions}
	\label{sec:preldef}
	
	To analyze the IIA testing problem, we require several new perspectives on how to encode the structure of a decision maker's behavior (IIA or not) as a mathematical object. The standard object of study for analyzing the discrete choices of an agent is the collection of probability distributions that describe an agent's choices over every subset of a universe of alternatives, together with the separate probability distribution that controls the frequencies of which sets are being queried. As a first step, in Section~\ref{subsection:choice_system} we introduce a new perspective on this union of objects as a single constrained high-dimensional discrete (i.e., categorical, multinomial) distribution. With this perspective, we are then able to define the property of satisfying IIA as a set of (non-linear) constraints on this high-dimensional object.
	
	A second perspective that we introduce in Section~\ref{subsection:comparison_graph} is the idea of a comparison graph that captures the structural richness of $\mathcal{C}$, the collection of unique comparison sets in the dataset of interest. This comparison graph can capture the structure of arbitrary sets with arbitrary sizes, and differs from the standard study of ``comparison graphs'' from pairwise comparisons \cite{ford1957solution}. We take this new perspective because properties native to this graph---chiefly, the average cycle length and a concept we term the {\it cycle dispersion index}---map directly to the fundamental limitations of testing IIA for a given collection $\mathcal{C}$.
	
	\subsection{Choice systems}\label{subsection:choice_system}
	Let $\mathcal{X}$ be a finite set of $n$ \textit{items} with generic element $x \in \mathcal{X}$, and let $\mathcal{C} \subseteq 2^\mathcal{X}$ denote the observed space, a collection of $m \geq n$ unique unordered subsets $C \subseteq \mathcal{X}$ of size 2 or greater. A decision maker who is presented with a \textit{choice set} $C \in \mathcal{C}$ chooses one item from the set. We let $P_{x,C}$ denote the probability that $x$ is chosen from $C$, $\forall x \in C$, and let $w(C)$ denote the probability of seeing choice set $C$, $\forall C \in \mathcal{C}$. The object $\mathcal{C}$ should be treated as a sample frame---the regime of sets of interest, or the regime of sets for which a notion of error is important---determined ahead of time (as opposed to being defined in terms of the sets so far observed in some dataset). That is, we are operating in a fixed-design setting, where $\mathcal{C}$ does not depend on, e.g., prior observations. 
	
	Our test is focused on a given dataset $\mathcal{D}_N = \lbrace (x_j, C_j) \rbrace_{j=1}^N$ that results from a decision maker making choices: a datapoint $j$ represents a decision scenario, and contains $C_j$, the choice set provided in that decision, and $x_j \in C_j$, the item chosen from that choice set. We assume each datapoint is an independent sample from the joint distribution over choices sets and choices:
	\begin{align*}
	Pr[(x_j, C_j) = (x, C)] = w(C) \cdot P_{x,C}, \ \ \forall j.
	\end{align*}
	That is, a datapoint is a sample from some discrete distribution over $d$ possible (choice, choice set) pairs, where $d = \sum_{C \in \mathcal{C}} |C|$. We refer to this $d$-dimensional discrete distribution as a \textit{choice system} and denote it using $q$. 
	Let $\Pall$ denote the collection of all choice systems on the collection $\mathcal{C}$, i.e., the collection of all $d$-dimensional discrete distributions. The (choice, choice set) tuples $(x,C)$ can be used to intuitively index this distribution, so we will generally use $(x,C)$ in place of $1,\ldots,d$ when addressing a specific position in a specific choice system $q$.

	Our definition of a choice system differs from Falmagne's \cite{falmagne1978representation,barbera1986falmagne} related definition of a {\it system of choice probabilities} $\{P_{x,C}\}_{\forall C \subseteq \mathcal X, \forall x \in C}$. Our definition of a system explicitly includes the probability distribution over sets, $w(C)$, $\forall C \in \mathcal X$, which Falmagne's concept does not. Falmagne's definition can be thought of as studying choices separated from an exogenously defined collection of choice sets. Further, because our choice systems are built to span only observed $C \in \mathcal{C}$, under our definition all choice systems have a positive set probability $w(C)>0$ for all $C \in \mathcal{C}$. Without the assumption that $w(C)>0$ for all $C \in \mathcal{C}$, the worst case error of any test of IIA would be trivially large, since it has no chance of measuring a violation of IIA within a set $C$ that it can not observe. 
	
	\subsection{Choice systems and IIA}
	The Independence of Irrelevant Alternatives (IIA) assumption constrains the probabilities $P_{x,C}$ to satisfy the following ratio for any $C \subset \mathcal{X}$: 
	\begin{align*}
	\dfrac{P_{x,\lbrace x, y \rbrace \cup C}}{P_{y,\lbrace x, y \rbrace \cup C}} = \dfrac{P_{x,\lbrace x, y \rbrace}}{P_{y,\lbrace x, y \rbrace}}.
	\end{align*}
	As derived by Luce in \cite{luce1959individual}, the IIA assumption consequently admits what Luce called a ratio representation,
	\begin{align*}
	P_{x,C} = \frac{\gamma_x}{\sum_{z \in C} \gamma_z},
	\end{align*}
	for some $\gamma \in \mathbb{R}_+^n$. Clearly, $\gamma$ is scale invariant. Setting a scale such that $\sum_{z \in \mathcal{X}} \gamma_z = 1$ gives the interpretation that $\gamma_z = P_{z, \mathcal{X}}$.
	That is, this interpretation says that an IIA choice system can be described by just knowing $w(C)$, $\forall  C \in \mathcal{C}$, and $P_{x,\mathcal{X}}$, $\forall x \in \mathcal X$. 
	
	We let $\Piia \subset \Pall$ denote the collection of all choice systems that satisfy IIA. Essentially, then, where $\Pall$ is the space of all $d$-dimensional discrete distributions, $\Piia$ is a family of (non-linearly) constrained $d$-dimensional discrete distributions. We emphasize that $\Piia$ is non-convex, and that the convex hull of $\Piia$ is dense in $\Pall$; the latter fact we show in the appendix (Fact~\ref{cvx_hull_dense}).
	
	\subsection{Separation from IIA}
	
	Let $\Mdelta \subset \Pall$ denote the set of non-IIA choice systems that are $\delta$-separated, in total variation distance, from the set of IIA choice systems $\Piia$, for a given collection of choice sets $\mathcal C$:
	\begin{align}
	\Mdelta = \{q : q \in \Pall, \ \inf_{p \in \Piia} \tv{q-p} \geq \delta\}.
	\end{align}
	Intuitively, $\delta$ describes the size of the \textit{indifference zone} \cite{lehmann2005testing} between the null and the alternative hypothesis, beyond which false acceptance becomes a serious error and ought to be limited. Since IIA and ``not IIA'' denote two contiguous sets, the separation makes the division of parameters sharp, and helps define the following testing problem, the central problem of this work:
	\begin{align}
	\label{eq:iiatest}
	\begin{cases} 
	H_0: (x, C) \sim p^N & \text{for some unknown } p \in \Piia \\
	H_1: (x,C) \sim q^N & \text{for some unknown } q \in \Mdelta.
	\end{cases}
	\end{align}
	
	Given the hypotheses in ~\eqref{eq:iiatest}, we define a test $\phi$ as a map, $\phi : \{ (x_1,C_1),...,(x_N,C_N) \} \mapsto \{0,1\}$ from the dataset to a decision to reject the null $H_0$. Our choice of total variation to define the separation of the testing problem is motivated by the distance metric's natural interpretation as describing the maximal gap in probability between two distributions over all possible events. Moreover, the choice is aligned with many other recent analyses of testing for discrete distributions~\cite{paninski2008coincidence,wei2016sharp,valiant2017automatic} and thus enables a direct comparison of the difficulty of testing IIA with the difficulty of testing other properties such as identity or monotonicity. 
	
	Several other measures of distance for the separation, e.g., Hellinger distance or Kullback-Liebler divergence, may also be reasonable~\cite{acharya2015optimal,daskalakis2018distribution}. The Hellinger distance metric is closely related to total variation distance and has been shown in certain testing scenarios to produce rates identical to those resulting from total variation distance~\cite{daskalakis2018distribution}. We conjecture this is also the case for the IIA testing problem, and leave such an exploration for future work. The Kullback-Liebler divergence is a weaker measure of separation; although it has been shown to lead to trivial (infinite) minimax risk when used in general problems of identity testing in discrete distributions, this concern does not arise for the case for IIA testing\footnote{Whereas in identity testing, two distributions can be arbitrarily close while having infinite KL divergence, the KL divergence to the closest point within IIA is always upper bounded by a finite quantity strictly decreasing with total variation distance.}. The measure is, however, unusual and difficult to interpret. As stated before, consider the $d$-dimensional uniform distribution and distributions $\epsilon$ far away in total variation distance; the KL divergence is $\mathcal{O}(d)$ larger than the TV distance when the distributional difference lies within a constant number of entries (as opposed to being spread evenly across all of the entries). For these reasons, we focus on total variation distance in this work. 

	\subsection{Comparison incidence graphs}\label{subsection:comparison_graph}
	To represent the comparison structure imposed by $\mathcal{C}$, we consider an undirected bipartite graph $G_\mathcal{C}$ with $n$ nodes for each item in $\mathcal{X}$ and $m$ nodes for each set in $\mathcal{C}$. Edges in this graph are drawn from the item nodes to the set nodes to indicate membership. That is, the nodes corresponding to each $C \in \mathcal{C}$ have edges extending to the items contained in $C$. This bipartite graph can also be thought of as representing a hypergraph with $n$ nodes, one for each item in $\mathcal{X}$, and $m$ hyperedges, one for each set in $\mathcal{C}$. The graph $G_\mathcal{C}$ is then the incidence graph of this hypergraph, and contain $d$ edges. As a result we call $G_\mathcal{C}$ the {\it comparison incidence graph}, to avoid confusion with other objects called comparison graphs in the study of pairwise comparisons \cite{ford1957solution,negahban2016rank}. 
	
	Throughout this work, we will assume that $G_\mathcal{C}$ is connected, meaning that there is a path in the bipartite graph from every item $i$ to every item $j$.
	
	This assumption corresponds to the necessary and sufficient condition for inference of an IIA choice system first established for pairs by Ford~\cite{ford1957solution} and later more generally by Hunter~\cite{hunter2004mm}. Moreover, we note that the requirement that $m \geq n$, stated earlier, is a sufficient condition on $\mathcal{C}$ to guarantee that there are cycles in $G_\mathcal{C}$\footnote{The necessary and sufficient condition to guarentee cycles in a connected $G_\mathcal{C}$ is $d \geq m + n$. Though our analysis still applies in this regime, we prefer the stronger condition $m \geq n$ to produce shorter analyses and cleaner expressions at the cost of sharpness when $m < n$.}. As we shall see, cycles are critically important: a $G_\mathcal{C}$ without cycles corresponds to a collection $\mathcal{C}$ for which all choice systems are IIA, rendering a test for IIA meaningless. 
	
	We call a decomposition of $G_\mathcal{C}$ into a set of $s$ simple cycles a {\it cycle decomposition}, and denote a cycle decomposition by $\sigma$. We define two functions of $\sigma$ that appear in the lower bound: $\mu(\sigma) = \frac{d}{|\sigma|}$, the \textit{average cycle length} of cycles in the decomposition and $\alpha(\sigma) = \frac{1}{d}\sum_{\sigma_i \in \sigma} |\sigma_i|^2$, the \textit{cycle dispersion index}\footnote{Although the index of dispersion traditionally refers to a ratio of variance and mean, we abuse the term slightly here to refer to a ratio of the cycle lengths' second moment and the cycle lengths' mean.} of cycles in the decomposition. We will elaborate on these terms further in Section \ref{sec:bounding}. 
	
	For simplicity of presentation of the results and proofs, for the remainder of this work we only consider the special case of collections $\mathcal{C}$ for which $|C|$ is even, $\forall C \in \mathcal{C}$, and where every item appears an even number of times across the subsets in $\mathcal{C}$. Under this restriction, we note that every node in $G_\mathcal{C}$ (on both sides of the bipartite graph) has even degree. As a consequence, $G_\mathcal{C}$ is Eulerian, which will be critical for our analysis. 
	We note that a graph $G_\mathcal{C}$ with odd degrees can be made Eulerian by removing simple paths between odd-degree nodes. While we do not explicitly extend our lower bounds beyond $\mathcal C$ for which $G_\mathcal{C}$ is Eulerian, a modest exercise in bookkeeping relaxes the requirement of even set sizes and even item appearances. 
	
	\section{Lower bound on the minimax risk}
	\label{sec:lowerbounds}
	
	Given a dataset $\mathcal D_N = \lbrace (x_j, C_j) \rbrace_{j=1}^N$ of choices, consider the testing problem, as stated in \eqref{eq:iiatest}, of the IIA property against the alternative of an arbitrary choice system that is $\delta$-separated from IIA. We seek to lower bound the minimax risk of this testing problem. The maximum risk of any test $\phi$ is a useful quantity that demonstrates how poor the test could be at finding IIA violations without further prior knowledge about the structure of the violations. Because tests can provide high power over certain types of violations while providing no power over others, it provides a level standard of comparison across tests by studying each test's worst case scenario \cite{balakrishnan2018hypothesis}.
	
	A minimax lower bound then bounds the error of the best possible test in its worst case, from below. The bound is given as a function of the number of samples $N$, the separation $\delta$ between the null and the alternative, and problem parameters. It is a statement about the fundamental possibility (or limitation) of effective tests for the testing problem. A lower bound on the minimax risk can also be restated as a lower bound on the sample complexity---the number of samples required to achieve a probability of error for a given separation---and equivalently, as a lower bound on the testing radius---the separation required to perform a test with a certain probability of error given a specific number of samples. 
	
	\subsection{Minimax risk preliminaries}
	The minimax risk $R_{N,\delta}(\Piia)$ of the IIA testing problem, where risk is used in a uniform sense between the null and the alternative, is then
	\begin{align*}
	R_{N,\delta}(\Piia) = \inf_{\phi} \sup_{p \in \Piia, q \in \Mdelta} \frac{1}{2}p^N(\phi(\mathcal{D}_N) = 1) + \frac{1}{2}q^N(\phi(\mathcal{D}_N) = 0),
	\end{align*}
	where the $\inf$ is taken over any test $\phi$. The testing radius $\delta_N(\Piia)$ is then be defined as 
	\begin{align*}
	\delta_N(\Piia) \defeq \inf \lbrace \delta \mid R_{N,\delta}(\Piia) \leq a \rbrace,
	\end{align*}
	where $a$ is any constant less than $1/2$ (sometimes arbitrarily set to $a=1/3$). The testing radius describes the smallest size of indifference zone that allows hypotheses to be uniformly distinguishable with probability $1-a$. 
	Finally, the sample complexity $N_\delta(\Piia)$ is defined similarly as
	\begin{align*}
	N_\delta(\Piia) \defeq \inf \lbrace N \mid R_{N,\delta}(\Piia) \leq a \rbrace. 
	\end{align*}
	
	For both the sample complexity and testing radius quantities, we at times use $\GtrSim$, followed by a expression, to denote that the quantities scale at at least a constant times that expression. As an alternative objective to lower bounding the minimax risk $R_{N,\delta}(\Piia)$, we provide in the appendix (Section~\ref{app:alphatest}) a reformulation of risk bounding as a problem centered around level-$\alpha$ tests, with the corresponding parallel statement to the risk-centered Theorem~\ref{thm:lower_bound} given below. 
	
	\subsection{Statement of main result}
	
	Given the preliminaries involved in defining our test and the concepts of minimax risk introduced above, we can now state our main result. We briefly interpret the result, but the proof of the theorem (given in Section~\ref{sec:thmproof}) follows only after we go through two major steps: relaxing our test to a simpler test (Section~\ref{sec:relax}) and introducing lower bounds on properties of the relaxed test (Section~\ref{sec:bounding}).
	
	\begin{theorem}\label{thm:lower_bound}
		Up to some constant $c_1$ and properties $\mu(\sigma)$ and $\alpha(\sigma)$ of any cycle decomposition $\sigma$ of the comparison incidence graph $G_\mathcal{C}$, the minimax risk $R_{N,\delta}(\Piia)$ is lower bounded as 
		$$
		\frac{1}{2} - 
		\frac{1}{4} 
		\Big(\exp\Big(
		\frac{8\mu(\sigma)^4\alpha(\sigma)N^2 \delta^4}{d}\Big) -1 
		\Big)^{\frac{1}{2}} 
		\leq R_{N,\delta}(\Piia).
		$$ 
		The testing radius $\delta_N(\Piia)$ and sample complexity $N_\delta(\Piia)$ then scale as at least
		$$
		\delta_N(\Piia) \GtrSim \frac{d^{\frac{1}{4}}}{\mu(\sigma)\alpha(\sigma)^{\frac{1}{4}}\sqrt{N}}, \ \ \ \ 
		N_\delta(\Piia) \GtrSim \frac{\sqrt{d}}{\sqrt{\mu(\sigma)^4\alpha(\sigma)}\delta^2}.$$
		\end{theorem}
	
	Examining the lower bound, it is clear that if the quantity $\frac{8\mu(\sigma)^4\alpha(\sigma)N^2 \delta^4}{d}$ is small, the minimax risk is bounded away from 0, and no uniformly consistent test exists. Keeping the quantity small then immediately determines a lower bound on the scaling, up to constants, of the testing radius and sample complexity. 
	
	The dependence of the lower bound on the comparison incidence graph properties $\mu(\sigma)$ and $\alpha(\sigma)$ is a feature unique to the IIA testing problem not found in previous general testing problems with discrete distributions. As we will illustrate more clearly in Section~\ref{sec:specifics}, we find that for well-behaved, uniform, and dense comparison incidence graphs (e.g., for the collection of all pairwise comparisons), both $\mu(\sigma)$ and $\alpha(\sigma)$ can be set to constants independent of other problem parameters (and then incorporated into a constant $c$). 
	
	We will further show, through an extremal analysis, that $\mu(\sigma)$ and $\alpha(\sigma)$ are at most $2n$ for any comparison incidence graph, and approximately $c \log(n)$ for sufficiently dense comparison incidence graphs, where $n$ is the number of alternatives and $c$ a constant. This analysis yields a global lower bound on the minimax risk, one we state as a corollary to the main theorem in Section \ref{sec:thmproof}, that replaces $\mu(\sigma)$ and $\alpha(\sigma)$ in the risk lower bounds with expressions only in terms of $n$ and $d$.
	
	Although our main lower bound broadly addresses the question of testing IIA for a given collection $\mathcal{C}$, IIA is often defined in the discrete choice literature as a single property of the complete choice system of $n$ items. By this definition, any true test of IIA must encompass all subsets of an item universe, a setting where $d = n 2^{n-1}$. For this case, we obtain the following bound, derived in Section~\ref{sec:specifics}:
	$$
	R_{N,\delta}(\Piia) \geq \frac{1}{2} - \frac{1}{4}\Big(\exp\Big(\frac{c \log(n)^5N^2 \delta^4}{n2^{n-1}} \Big) - 1\Big)^{\frac{1}{2}},
	$$
	with associated testing radius $\delta_N(\Piia)$ and sample complexity $N_\delta(\Piia)$ that scale as at least
	$$
	\delta_N(\Piia) \GtrSim
	\frac{n^{\frac{1}{4}}2^{\frac{n}{4}}}{\log(n)^{\frac{5}{4}}\sqrt{N}}, 
	\ \ \ \ 
	N_\delta(\Piia) \GtrSim 
	\frac{\sqrt{n}2^{\frac{n}{2}}}{\log(n)^{\frac{5}{2}}\delta^2}.$$
	The sample complexity demonstrates that testing IIA, when treated as a general property of a choice system, requires at least samples \textit{exponential} in the number of items $n$. We can thus conclude that testing IIA, when treated as a property of a complete choice system, is hopelessly intractable for large collections.
	
	As stated earlier, the flexibility of our main lower bound allows for more refined analyses for specific collections $\mathcal{C}$, and yields room for optimism (that there may exist reasonably efficient tests) in a variety of special cases. Table ~\ref{table:all_lower_bounds} shows the specific lower bound for the ``all subsets'' setting, in addition to highlighting these special cases where alternative arguments about $\mu(\sigma)^4\alpha(\sigma)$ can be applied to the general bound. Section~\ref{sec:specifics} gives a detailed derivation and discussion of these settings. 
	
	\begin{table}[]
		\caption{
			Lower bounds on the minimax risk $R_{N,\delta}$ of the IIA testing problem, as defined in \eqref{eq:iiatest}, and lower bounds on the scaling rates of the associated testing radius $\delta_N$ and sample complexity $N_\delta$. For a given collection $\mathcal C$, $n$ is the number of items in the choice problem, $N$ is the number of samples, $d = \sum_{C \in \mathcal{C}} |C|$, and $\sigma$ is any cycle decomposition of the comparison incidence graph $G_{\mathcal C}$. Different assumptions about the structure of $\mathcal C$ invite different guaranteed properties of some cycle decomposition $\sigma$; thus the different bounds reflect differences in $d$ but also differences in guarantees about $\mu(\sigma)$ and $\alpha(\sigma)$ for some $\sigma$. \\ 
		}
		\small
		\begin{tabular}{|l|c|c|c|}
			\hline
			Structure of $\mathcal{C}$
			& 
			$R_{N,\delta}(\Piia)$
			& 
			$\delta_N(\Piia)$
			& 
			$N_\delta(\Piia)$ 
			\\ \hline
			
			General
			&
			$\geq \frac{1}{2} - 
			\frac{1}{4} 
			\Big(\exp\Big(
			\frac{8\mu(\sigma)^4\alpha(\sigma)N^2 \delta^4}{d}\Big) -1 
			\Big)^{\frac{1}{2}}$
			& 
			$\GtrSim \frac{d^{\frac{1}{4}}}{\mu(\sigma)\alpha(\sigma)^{\frac{1}{4}}\sqrt{N}}$
			& 
			$\GtrSim \frac{\sqrt{d}}{\sqrt{\mu(\sigma)^4\alpha(\sigma)}\delta^2}$
			
			\\ \hline
			All subsets, $d=n2^{n-1}$
			&
			$\geq \frac{1}{2} - \frac{1}{4}\Big(\exp\Big(\frac{c \log(n)^5N^2 \delta^4}{n2^{n-1}} \Big) - 1\Big)^{\frac{1}{2}}$
			&
			$\GtrSim \frac{n^{\frac{1}{4}}2^{\frac{n}{4}}}{\log(n)^{\frac{5}{4}}\sqrt{N}}$
			& 
			$\GtrSim \frac{\sqrt{n2^{n}}}{\log(n)^{\frac{5}{2}}\delta^2}$
			\\ \hline 
			All pairs, $d=n(n-1)$
			& 
			$\geq \frac{1}{2} - \frac{1}{4}\Big(\exp\Big(\frac{c N^2 \delta^4}{n(n-1)} \Big) - 1\Big)^{\frac{1}{2}}$
			&
			$\GtrSim \frac{\sqrt{n}}{\sqrt{N}}$
			&
			$\GtrSim \frac{n}{\delta^2}$
			\\ \hline 
			Single big cycle, $d=2n$
			&
			$\geq \frac{1}{2} - \frac{1}{4}\Big(\exp\Big(c n^4N^2 \delta^4 \Big) - 1\Big)^{\frac{1}{2}}$
			&
			$\GtrSim \frac{1}{n\sqrt{N}}$
			&
			$\GtrSim \frac{1}{n^2\delta^2}$
			\\ \hline
		\end{tabular}
		\label{table:all_lower_bounds}
	\end{table}
	
	\section{Testing relaxation}
	\label{sec:relax}
	Rather than directly considering the IIA test stated in \eqref{eq:iiatest}, we will now take two concrete steps to introduce an analytically simpler test that is not statistically harder than the IIA test. By lower bounding the minimax risk of the simpler test, we obtain a lower bound on the harder IIA test. Two relaxation steps now follow, the first of which restricts $H_0$ to a single element of $\Piia$, the uniform distribution $p_0$, and the second of which restricts $H_1$ from $\Mdelta$ to a carefully constructed assortment of perturbations of $p_0$ that we can show all reside in $\Mdelta$.
	
	For the first simple step, consider the testing problem:
	\begin{align}
	\label{eq:test2}
	\begin{cases} 
	H_0: (x, C) \sim p^N & \text{for } p = p_0 \\
	H_1: (x,C) \sim q^N & \text{for some unknown } q \in \Mdelta,
	\end{cases}
	\end{align}
	where $p_0$ is the uniform distribution, $p_{0,(x,C)} = 1/d, \forall (x,C)$ (recalling that each discete distribution is $d$-dimensional, where $d$ is the sum of the cardinalities of the choice sets in a given $\mathcal C$). Since $p_0 \in \mathcal{P}_0$, the problem is simpler than the original problem, and any lower bound on the performance of this problem carries over to the original problem of interest. 
	
	Now consider a discrete distribution $\qbe$ that is a perturbation of the uniform distribution of the form:
	\begin{align*}
	\qbe((x,C)) = \frac{1}{d} + \frac{\epsilon b_{(x,C)}}{d}, 
	\end{align*}
	\begin{align*}
	\epsilon \in [0,1] \ 
	b \in \lbrace -1, 1 \rbrace^d,
	\ 
	\sum_{y \in C} b_{y,C} = 0, \forall C \in \mathcal C,
	\text{ and } 
	\sum_{\substack{C \in \mathcal{C}\\ C \ni x}} b_{x,C} = 0, \forall x.
	\end{align*}

	That is, $b$ perturbs the uniform distribution such that there is no net translation over any set, or over all of the appearances of an item over sets. Our goal is to construct perturbations $\qbe$, selecting $\epsilon$ and $b$ that land in $\Mdelta$, meaning they are at least $\delta$ away (in TV distance) from $\Piia$ but at the same time still correspond to valid probability distributions (i.e., within the $d$-simplex). We are starting our journey (in the direction of $b$) at the uniform distribution $p_0$, and we want to show that we always land at least $\delta$ away from {\it any} distribution in $\Piia$.
	
	Let $\Bset$ be any collection of $M=|\Bset|$ vectors $b$ satisfying the above conditions for the given collection $\mathcal C$. That is, 
	\begin{align}
	\label{eq:bset_full}
	\Bset \subseteq 
	\left \{
	b \in \lbrace -1, 1 \rbrace^d
	:
	\sum_{y \in C} b_{y,C} = 0, \forall C \in \mathcal C,
	\ \ 
	\sum_{\substack{C \in \mathcal{C}\\ C \ni x}} b_{x,C} = 0, \forall x
	\right \}.
	\end{align}
	We define $\bar{q}_\epsilon = \frac{1}{M} \sum_{b \in \Bset} \qbe$ as the mixture distribution over the $M$ distributions in $\Bset$. Samples from $\bar{q}_\epsilon$ would be drawn marginally, that is, some $\qbe$ would be chosen uniformly from the $M$ possible distributions, and samples would then come from that $\qbe$. Consider then the testing problem:
	\begin{align}
	\label{eq:test3}
	\begin{cases} 
	H_0: (x, C) \sim \mathbb{P}_0 = 
	{p_0}^N \\
	H_1: (x,C) \sim \mathbb{P}_1 = 
	{\bar{q}_{\epsilon} }^N.
	\end{cases}
	\end{align}
	The difference between the original test in \eqref{eq:iiatest} and the test in \eqref{eq:test3} is illustrated in Figure~\ref{fig:piia}. 
	
	\begin{figure}[t]
		\centering
		\includegraphics[height=35mm]{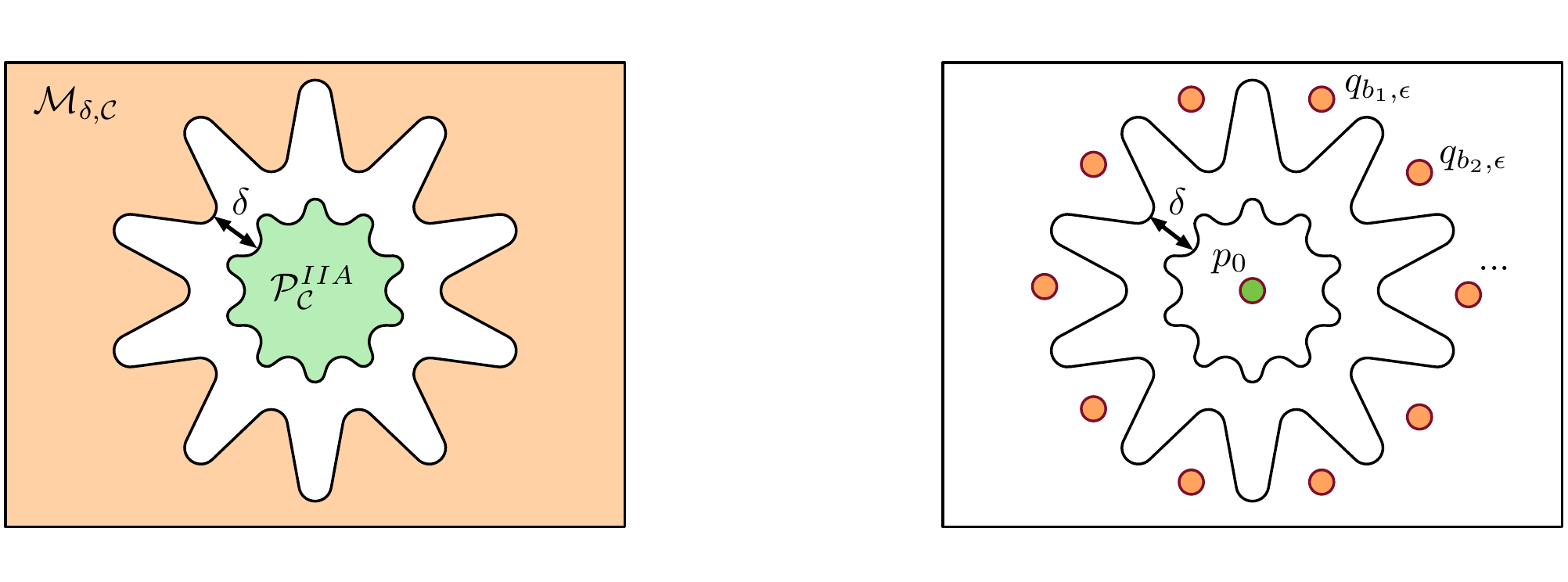}
		\caption{
			A schematic illustration of the test simplification procedure employed to lower bound the minimax risk. The illustration on the left corresponds to the IIA test in \eqref{eq:iiatest}, while the illustration on the right corresponds to the easier test in \eqref{eq:test3}, replacing $\Piia$ with $p_0$, the uniform distribution, and $\Mdelta$ with the collection of perturbations $\{\qbe\}_{b \in \mathcal{B}_\mathcal{C}}$.
			}
		\label{fig:piia}
		\vspace{-4mm}
	\end{figure}
	
	We now wish to identify a (large) set of $\qbe$ where we can lower bound the distance from each $\qbe$ back to $\Piia$. Our strategy will involve considering $b$ where we can partition the indices into subsets and show that each subset has a lower bound on its contribution to the distance back to $\Piia$. We then show that every partition will contribute at least a small distance to the total TV distance from IIA. 
	
	\xhdr{Perturbations and the comparison incidence graph}  
	Consider the undirected bipartite graph $G_\mathcal{C}$ defined in Section~\ref{subsection:comparison_graph}, and recall that by assumption every item appears an even number of times across sets in $\mathcal{C}$, and every set is of even size. We then have that every node in $G_\mathcal{C}$ is of even degree---hence, the graph is Eulerian. 
	
	Our key step in constructing favorable perturbations $b$ will be to treat the values in the vector $b$ as directing the edges in the undirected graph $G_\mathcal{C}$. Specifically, if $b_{(x,C)} = 1$, the edge is directed to be from the item-node for item $x$, to the set-node for set $C$, and if $b_{(x,C)} = -1$ the edge has the opposite direction. We now restate the sum constraints on a $b$ vector:
	\begin{align*}
	\sum_{\substack{C \in \mathcal{C}\\ C \ni x}} b_{x,C} = 0 \ \ \forall x \in \mathcal{X}, && \sum_{y \in C} b_{y,C} = 0 \ \ \forall C \in \mathcal{C}.
	\end{align*}
	These constraints can be interpreted as requiring the \textit{in-degree} of each node in $G_\mathcal{C}$ to equal its \textit{out-degree}. We then observe that finding vectors $b$ that satisfy these constraints is equivalent to finding {\it Eulerian orientations} of the graph $G_\mathcal{C}$. Thus, $\Bset$ is simply any collection of $M$ vectors that each correspond to a unique Eulerian orientation of $G_{\mathcal C}$. While there can be many Eulerian orientations of $G_\mathcal{C}$, our hope is to work with a specially structured collection $\Bset$ such that the $\qbe$ can be shown to land in $\Mdelta$ for all $b \in \Bset$.
	
	\xhdr{Cycle decompositions and orientations of Eulerian bipartite graphs} To find a structured collection of Eulerian orientations of $G_\mathcal{C}$, we begin by examining the cycle decompositions of $G_\mathcal{C}$.
	We first raise the simple fact that every Eulerian graph can be decomposed into edge-disjoint simple undirected cycles. We remind the reader that we call a decomposition of $G_\mathcal{C}$ into a set of $s$ simple cycles a {\it cycle decomposition}. There are ostensibly many different ways to select edge-disjoint cycles of $G_\mathcal{C}$, and we denote the collection of all possible edge-disjoint cycle decompositions of $G_\mathcal{C}$ by $\Sigma_{\mathcal{C}}$.
	
	\begin{figure}[t]
		\centering
		\includegraphics[height=55mm]{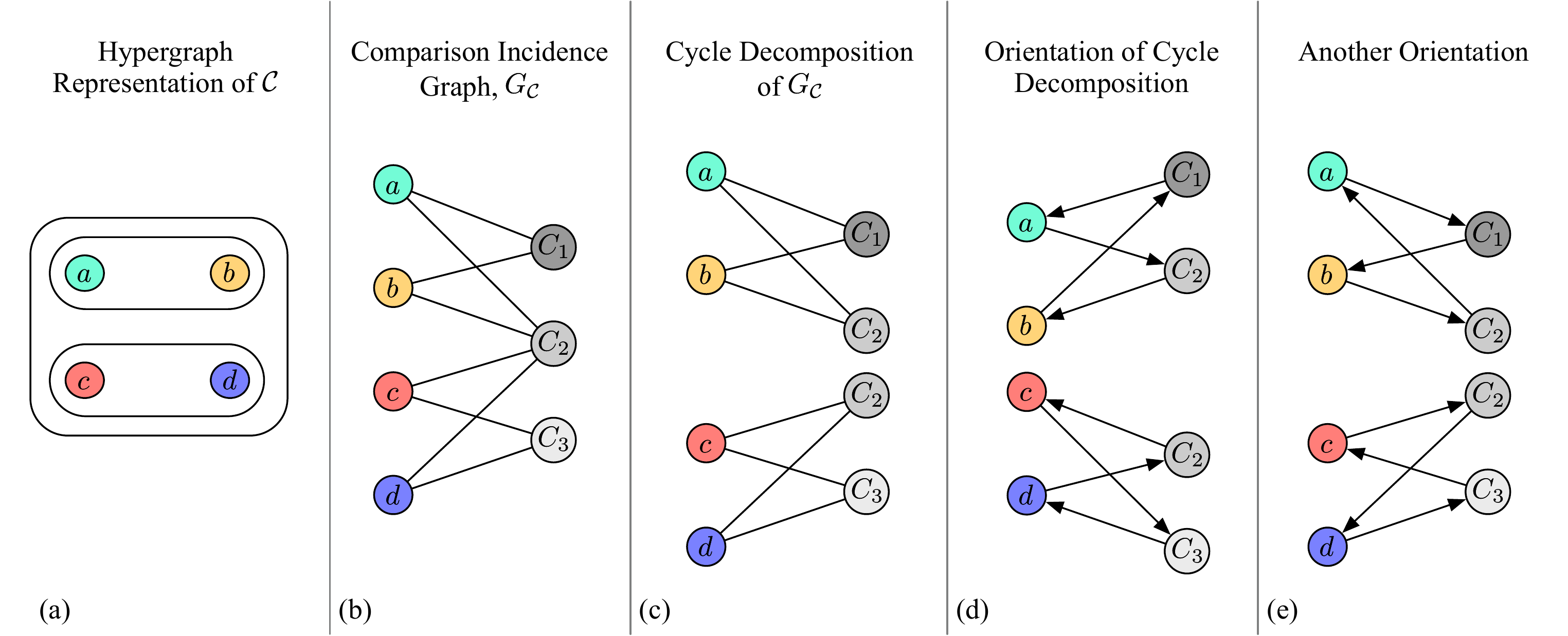}
		\caption{
			A diagram illustrating the steps undertaken to construct Eulerian orientations for a simple example collection $\mathcal C$. (a) The hypergraph of the collection under study, $\mathcal{C} = \{\{a,b\},\{c,d\},\{a,b,c,d\}\}$. (b) The comparison incidence graph $G_{\mathcal C}$ corresponding to $\mathcal C$. (c) One possible cycle decomposition of $G_\mathcal{C}$, consisting of two undirected cycles. (d,e) Two different Eulerian orientations of the given cycle decomposition. In the first orientation, (d), both cycles are directed counterclockwise. In the second, (e), both cycles are directed clockwise. Mixtures of these two rotations produce two other orientations of the given cycle decomposition, not shown, yielding a total of four orientations. 
		}
		\label{fig:gc}
		\vspace{-4mm}
	\end{figure}
	A cycle decomposition $\sigma \in \Sigma_\mathcal{C}$ is a set of cycles forming a feasible edge-disjoint cycle decomposition of $G_\mathcal{C}$, and an element $\sigma_i \in \sigma$ denotes the $i$th cycle in $\sigma$. We use $|\sigma|$ to denote the number of cycles in the decomposition $\sigma$, and $|\sigma_i|$ to denote the length of each cycle. Since all cycles of bipartite graphs have even length, we will write $|\sigma_i|=2k_i$, where $k_i$ is the number of item-nodes in the cycle. 

	Lastly, observe that we can always find a cycle decomposition of $G_\mathcal{C}$ such that every cycle has size at most $2n$. This observation follows from a simple application of the pigeonhole principle: since $G_\mathcal{C}$ is bipartite, with $n$ item-nodes and $m$ set-nodes, any cycle of length greater than $2n$ must visit some item-node more than once; hence the cycle is not simple, and can be decomposed into two smaller edge-disjoint cycles. 
	
	\xhdr{From cycle decomposition to perturbations}
	We generate a structured collection of perturbation vectors $b$ by generating many distinct Eulerian orientations of a given cycle decomposition $\sigma$. We can independently direct the edges in each of the	cycles in $\sigma$ in one of 2 directions (clockwise or counterclockwise), and all of these collections of directions correspond to valid $b$ vectors. Thus, we have found $M=2^{|\sigma|}$ different Eulerian orientation for $G_\mathcal{C}$, each with a one-to-one correspondence to a valid perturbation vector $b$.
	
	We define $\Bset(\sigma)$ as the {\it specific} collection of perturbation vectors $b$ that orient a cycle decomposition $\sigma$, refining our generic definition of such collections of perturbations $\Bset$ in Equation~\eqref{eq:bset_full}:
	\begin{align}
	\label{eq:bset}
	\Bset(\sigma) = 
	\left \{
	b \in \lbrace -1, 1 \rbrace^d
	:
	\sum_{y \in C} b_{y,C} = 0, \forall C, 
	\ \ 
	\sum_{\substack{C \in \mathcal{C}\\ C \ni x}} b_{x,C} = 0, \forall x,
	\ \ 
	|b_\ell - b_{\ell+1} | =2, \forall \ell \in \sigma_i,
	\forall \sigma_i \in \sigma
	\right \}.
	\end{align}
	Here $(x,C)$ indexes the elements of the vector $b$ and we use $\ell$ to index sequential edges in a cycle $\sigma_i$. For each perturbation vector $b \in \Bset(\sigma)$, we can construct a discrete distribution $\qbe$ as described previously. We additionally define $\mu(\sigma) = \frac{d}{|\sigma|}$, the average cycle length.
	
	\xhdr{Bounding separation} 
	For any cycle decomposition $\sigma$ and its corresponding collection of vectors $b \in \Bset(\sigma)$, with corresponding distributions $\qbe$, we obtain the following result.
	\begin{lemma}\label{lemma:seperation}
		A test between $p_0^N$ and ${\bar{q}_{\epsilon} }^N$, for any $N$ and $\epsilon \geq 2 \mu(\sigma) \delta$ is not statistically harder than the original IIA test. In particular, $ \epsilon \ge 4n \delta \geq 2\mu(\sigma)\delta$, suffices for all $\sigma$.
	\end{lemma}
	\begin{proof}
		If we can show that every $\qbe \in \Mdelta$, certainly then the mixture distribution\footnote{We remind the reader that we use mixture distribution to mean a distribution from which samples are drawn marginally.} ${\bar{q}_{\epsilon} } \in \Mdelta$, and hence the new test in \eqref{eq:test3} is easier than the original test in \eqref{eq:iiatest}. From the definitions of choice systems and of TV distance, we see that for any $\qbe$,
		\begin{align*}
		\inf_{p \in \Piia} \tv{\qbe - p} = \inf_{\gamma \in \Delta_n, w \in \Delta_m} \frac{1}{2} \sum_{C \in \mathcal{C}} \sum_{j \in C} \Big|w_C\frac{\gamma_j}{\sum_{k\in C} \gamma_k} - \frac{1+\epsilon b_{(j,C)}}{d}\Big|,  
		\end{align*}
		where $n$ is the number of items and $m=|\mathcal{C}|$ is the number of unique sets.
		
		While it is easy to show that $\tv{\qbe - p_0} = \frac{\epsilon}{2}$, that is, all the $\qbe$ are exactly $\frac{\epsilon}{2}$ away from the point they were perturbed from, the non-convexity of $\Piia$ provides little intuition about the distance between $q_b$ and the closest other element in $\Piia$. See Figure \ref{fig:piia} for an illustration of the complications associated with distance from a non-convex set. 
		
		We do not exactly solve the problem of finding the minimum distance. Rather, we take advantage of the specific structure of $b$ to provide a lower bound on the problem's optimal value, the closest that any choice system satisfying IIA could be to $\qbe$. We proceed first through a relaxation:
		\begin{align}\label{relaxed_obj}
		\inf_{p \in \Piia} \tv{\qbe - p} \geq \inf_{\gamma \in \mathbb R^n_+, w \in \mathbb R^m_+} \frac{1}{2} \sum_{C \in \mathcal{C}} \sum_{j \in C} \Big|w_C \gamma_j - \frac{1+\epsilon b_{(j,C)}}{d}\Big|.
		\end{align}
		Here, two things happened: first, the lower bound follows from removing the unit simplex constraints, and second, since the denominator gamma terms are non-negative and only vary with every $C$, they maybe be absorbed into a unconstrained non-negative $w_C$ without changing the value of the optimal solution. 
		
		Next, before we consider the objective over an arbitrary $\mathcal{C}$, we first consider a very structured collection, a series of pair comparisons that form a ``cycle'' among unique items: e.g. $\mathcal{C}_\text{cycle} = \{i,j\}, \{j,k\}, \{k,l\}...\{z,i\}$. Say $\mathcal{C}_\text{cycle}$ has $k$ unique items. The comparison incidence graph $G_\text{cycle}$ corresponding to $\mathcal{C}_\text{cycle}$ is then bipartite with $k$ item-nodes, $k$ set-nodes, and $2k$ edges. The cycle decomposition of $G_\text{cycle}$ contains just one cycle, and the perturbations $b$ we consider for this graph are one of the two possible orientations of this single cycle. Without loss of generality, consider the ``clockwise direction''. 
		
		We again will not try to solve the problem in~\eqref{relaxed_obj} exactly; instead, we will show that it must fail to achieve zero error on at least one element of the sum, by virtue of the fact that the terms $w_C \gamma_j$ can not all simultaneously deviate in the same direction from the uniform point $\frac{1}{d}$ as $\qbe$ does. 
		Thus we want to show that it is not possible for {\it every} $w_C \gamma_j$ product to achieve a value above $\frac{1}{d}$ when $b_{(j,C)} = 1$ and a value below $\frac{1}{d}$ when $b_{(j,C)}=-1$. 
		
		We demonstrate this impossibility as follows: we follow the cycle, and show that while we can zero out most of the contributions to the loss, there is a waterbed effect that prevents us from zeroing out all the values. Begin, without loss of generality, at item $i$ in the first set (label it $C_1$), with all parameters starting at the middle $\frac{1}{\sqrt{d}}$ point. In order for $w_{C_1}\gamma_i$ to be higher than $\frac{1}{d}$, we must raise, without loss of generality, $w_{C_1}$ by some value $\delta_A > 0$. Then, in order for $w_{C_1}\gamma_j$ to be below $\frac{1}{d}$, we must lower $\gamma_j$ by some $\delta_B > \delta_A$. Then, for $w_{C_2}\gamma_j$ to be above $\frac{1}{d}$, we must raise $w_{C_2}$ by some $\delta_C > \delta_B$. We proceed this way, alternating between $w$ and $\gamma$, with each increase and decrease larger in magnitude than the previous to meet the sign requirements, until we arrive at the last set. Following the pattern, we have raised $w_{C_k}$ by some $\delta_Y > \delta_X > ... > \delta_A$ in order for $w_{C_k}\gamma_z$ to be below $\frac{1}{d}$, we need to raise $\gamma_1$ by $\delta_Z > \delta_Y > ... \delta_A$ in order for $w_{C_k}\gamma_i$ to be above $\frac{1}{d}$. But, since $\delta_Z > \delta_A$, we no longer obey the sign constraint for $w_{C_1}\gamma_i$. 
		
		The symmetry of the cycle generalizes the result of this demonstration: starting at any arbitrary node for any initial values, for either cycle orientation, attempting to satisfy every sign constraint is always met with failure. Thus, we conclude that not all signs can be satisfied by any setting of $w$ and $\gamma$. A straightforward consequence of not being able to satisfy all the signs is that at least one pair of values is at least $\frac{\epsilon}{d}$ away from its optimal value by not being able to achieve the sign of that pair. Thus, we have
		\begin{align*}
		\inf_{\gamma, w} \frac{1}{2} \sum_{C \in \mathcal{C}_\text{cycle}} \sum_{j \in C} \Big|w_C \gamma_j - \frac{1+\epsilon b_{(j,C)}}{d}\Big| \geq \frac{1}{2} \frac{\epsilon}{d}.
		\end{align*}

		After a lengthy exposition, we have only provided a crude sign-based lower bound for the special case of a simple cycle, still leaving open the question of a lower bound for general $\mathcal{C}$. The key next step is in accounting for the design of the $b$ perturbations. Each perturbation $b \in \Bset(\sigma)$, associated with each $\sigma \in \Sigma_\mathcal{C}$, are all created by toggling simple cycles clockwise and counterclockwise. Since TV distance is linearly separable in each entry, we may then partition the $d$ entries of the vector $b$ according to each simple cycle. We may then place a lower bound on the objective as follows: rather than having to share the vertex labels across all $d$ entries, we lower bound the objective by independently optimizing over each simple cycle: 
		\begin{align*}
		\inf_{\gamma, w} \frac{1}{2} \sum_{C \in \mathcal{C}} \sum_{j \in C} \Big|w_C \gamma_j - \frac{1+\epsilon b_{(j,C)}}{d}\Big| &= \inf_{\gamma, w} \frac{1}{2} \sum_{\sigma_i \in \sigma} \sum_{(j,C) \in \sigma_i} \Big|w_C \gamma_j - \frac{1+\epsilon b_{(j,C)}}{d}\Big|\\
		&\geq \frac{1}{2} \sum_{\sigma_i \in \sigma} \inf_{\gamma_\sigma, w_\sigma} \sum_{(j,C) \in \sigma_i} \Big|w_C \gamma_j - \frac{1+\epsilon b_{(j,C)}}{d}\Big|.
		\end{align*}
		Since the entries associated with a simple cycle $\sigma_i$ has a lower bound of $\frac{\epsilon}{2d}$, the lower bound on the objective is simply $\sum_{i \in [|\sigma|]} \frac{\epsilon}{2d} = \frac{\epsilon |\sigma|}{2d}$. The more cycles, the higher this lower bound is---thus, this quantity is at least $\frac{\epsilon}{2d} \frac{d}{2n} = \frac{\epsilon}{4n}$ even in the worst case. To summarize, we have shown that
		\begin{align*}
		\inf_{p \in \Piia} \tv{\qbe - p} \geq \frac{\epsilon |\sigma|}{2d} \geq \frac{\epsilon}{4n}
		\end{align*}
		
		Thus, if $\epsilon \geq \frac{2d}{|\sigma|}\delta = 2\mu(\sigma)\delta$, a sufficient condition for which is $\epsilon \geq 4n \delta$, then $\qbe \in \Mdelta$ for all $b \in \mathcal{B}_\mathcal{C}(\sigma)$. 
	\end{proof}
	
	\section{Lower bound on the simplified test}
	\label{sec:bounding}
	
	Now, we can focus on lower bounding the performance of the simplified hypothesis test of IIA given in \eqref{eq:test3}, using Le Cam's method \cite{yu1997assouad} as a starting point. Using $\gamma(\mathbb{P}_0, \mathbb{P}_1)$ to denote the average of type I and type II errors of the best possible test for distinguishing between the binary hypotheses $\mathbb{P}_0$ and $\mathbb{P}_1$ defined in equation \eqref{eq:test3}, a testing inequality owed to Le Cam \cite{yu1997assouad} states that
	\begin{align}
	\gamma(\mathbb{P}_0, \mathbb{P}_1) \geq \frac{1}{2} - \frac{1}{2}\tv{\mathbb{P}_0 - \mathbb{P}_1}.
	\end{align}
	Thus, the minimax risk for the testing problem \eqref{eq:test3} has a lower bound, and since that testing problem is simpler than the original problem, we have a minimax lower bound for the original problem. That is,
	\begin{align*}
	R_{N,\delta}(\Piia) \geq \frac{1}{2} - \frac{1}{2}\tv{\mathbb{P}_0 - \mathbb{P}_1}.
	\end{align*}
	What remains is meaningfully upper bounding $\tv{\mathbb{P}_0 - \mathbb{P}_1}$, the total variational distance between the hypotheses, in terms of the parameters of interest. 
	
	Consider that
	\begin{align*}
	||p-q||^2_{TV} \leq \frac{1}{4} \chi^2(p,q),
	\end{align*}
	where 
	\begin{align*}
	\chi^2(p,q) = \sum_{i \in [d]} \frac{p_i^2}{q_i} - 1 = \mathbb{E}_q\Big(\frac{p}{q}\Big)^2 - 1.
	\end{align*}
	If we can upper bound the $\chi^2$ distance between the two hypotheses, we then have an upper bound on the total variational distance. 
	
	We obtain the following bound between the uniform distribution $p_0$ and our mixture of perturbations $\qbe$, decomposed into two lemmas. The first lemma covers the aspects of the bounding exercise that are statistical, while the second lemma covers aspects that engage with the combinatorial structure of our set of perturbations derived from Eulerian orientations of the comparison incidence graph. In each case a low upper bound of the $\chi^2$ distance provides a high lower bound of the minimax risk of the IIA test. 
	
	\begin{lemma}
		\label{lemma:bound1}
		$$
		\chi^2(\mathbb{P}_1,\mathbb{P}_0) + 1 
		\le
		\frac{1}{M^2}\sum_{b,b' \in \Bset}\exp\Big(\frac{N \epsilon^2}{d}b^Tb'\Big) 
		,
		$$
		where $\Bset$ is a set of arbitrary perturbations $b$ of size $|\Bset|=M$ satisfying $b \in \lbrace -1, 1 \rbrace^d$, $\sum_{\substack{C \in \mathcal{C}\\ C \ni x}} b_{x,C} = 0, \forall x$, and $\sum_{y \in C} b_{y,C} = 0, \forall C$. 
	\end{lemma}
	Recall that $\Bset$, defined in Equation~\eqref{eq:bset_full}, is more general than $\Bset(\sigma)$, defined in Equation~\eqref{eq:bset}. The former set is not required to have any connection to a cycle decomposition and its orientations.
	
	\begin{proof}
		The proof of the lemma adapts procedures from prior work \cite{ingster1987minimax,paninski2008coincidence,wei2016sharp} to the particular setting of the testing problem in \eqref{eq:test3}. We begin by manipulating the various definitions involved in this distance, 
		\begin{align*}
		\chi^2(\mathbb{P}_1,\mathbb{P}_0) + 1 &=\mathbb{E}_{\mathbb{P}_0}\Big(\frac{d\mathbb{P}_1}{d\mathbb{P}_0}\Big)^2\\
		&=\mathbb{E}_{\mathbb{P}_0}\Big(\frac{\frac{1}{M}\sum_{b \in \Bset}\qbe^N}{{p_0}^N}\Big)^2\\
		&=\mathbb{E}_{\mathbb{P}_0}\frac{1}{M^2}\sum_{b,b' \in \Bset}\frac{\qbe^N \qbpe^N}{{(p_0}^N)^2}\\
		&=\mathbb{E}_{\mathbb{P}_0}\frac{1}{M^2}\sum_{b,b' \in \Bset}\frac{\qbe((x_1,C_1))\cdot\cdot\cdot \qbe((x_N,C_N)) \qbpe((x_1,C_1))\cdot\cdot\cdot \qbpe((x_N,C_N))}{p_0((x_1,C_1))^2\cdot\cdot \cdot p_0((x_N,C_N))^2}\\
		&=\frac{1}{M^2}\sum_{b,b' \in \Bset}\Big(\mathbb{E}_{p_0} \frac{\qbe \qbpe}{p_0^2}\Big)^N.
		\end{align*}
		Here the last equality follows from the fact that the $(x_i,C_i)$ are i.i.d., and hence the expectation of the product is the product of the expectations, which are all the same. Now we take a closer look at the term inside the exponentiation:
		\begin{align*}
		\mathbb{E}_{p_0} \frac{\qbe \qbpe}{p_0^2} &= \sum_{C \in \mathcal{C}} \sum_{x \in C} \frac{\qbe((x,C)) \qbpe((x,C))}{p_0((x,C))}\\
		&= \sum_{C \in \mathcal{C}} \sum_{x \in C} \frac{1}{p_0((x,C))}\Big[p_0((x,C)) + \frac{\epsilon b_{(x,C)}}{d}\Big]\Big[p_0((x,C)) + \frac{\epsilon {b'}_{(x,C)}}{d}\Big]\\
		&= \sum_{C \in \mathcal{C}} \sum_{x \in C} \Big[p_0((x,C)) + \frac{\epsilon b_{(x,C)}}{d} + \frac{\epsilon {b'}_{(x,C)}}{d} + \frac{1}{p_0((x,C))}\Big(\frac{\epsilon b_{(x,C)}}{d}\Big)\Big(\frac{\epsilon {b'}_{(x,C)}}{d}\Big)\Big]\\
		&= 1+\sum_{C \in \mathcal{C}} \sum_{x \in C} \Big[\frac{1}{p_0((x,C))}\Big(\frac{\epsilon b_{(x,C)}}{d}\Big)\Big(\frac{\epsilon {b'}_{(x,C)}}{d}\Big)\Big]\\
		&= 1 + \frac{\epsilon^2}{d}b^Tb',
		\end{align*}
		where the second line follows from applying the definition of $\qbe$, the third from expanding the terms, and the fourth from the constraints $b$ has to satisfy, and probabilities summing to one, and the last line is simply rearranging the terms and recognizing that $p_0((x,C)) = \frac{1}{d}$ for all $(x,C)$. Substituting the final expression into the previous result, we have
		\begin{align*}
		\chi^2(\mathbb{P}_1,\mathbb{P}_0) + 1 &=\frac{1}{M^2}\sum_{b,b' \in \Bset}\Big(1 + \frac{\epsilon^2}{d}b^Tb'\Big)^N\\
		&\leq \frac{1}{M^2}\sum_{b,b' \in \Bset}\exp\Big(\frac{N \epsilon^2}{d}b^Tb'\Big).
		\end{align*}\end{proof}
		
	The remaining analysis applies only to sets of perturbations $\Bset(\sigma)$ derived from an Eulerian orientation $\sigma$ of $G_\mathcal{C}$. With such a collection of perturbations, we can bound the role of $b^Tb'$ for each pair of perturbations. Recall that $\Bset(\sigma)$ was defined in \eqref{eq:bset} as the collection of $2^{|\sigma|}$ Eulerian orientations associated with a particular cycle decomposition $\sigma$. We now additionally define a set of vectors $\mathcal{A}_\mathcal{C}(\sigma)$, such that for every element $b \in \mathcal{B}_\mathcal{C}(\sigma)$, $\mathcal{A}_\mathcal{C}(\sigma)$ contains a vector $a \in \lbrace -1, 1\rbrace^{|\sigma|}$ to denote the directions of the $|\sigma|$ cycles that created $b$. Moreover, we introduce for every $\sigma$ a quantity $\alpha(\sigma) = \frac{1}{d} \sum_{\sigma_i \in \sigma}|\sigma_i|^2$ that serves as a normalized measure of the ``dispersion'' of the cycle decomposition $\sigma$. With this notation, we proceed to the following result. 
	\begin{lemma}
		\label{lemma:bound2}
		\begin{align*}
		\chi^2(\mathbb{P}_1,\mathbb{P}_0) + 1 
		\ \ \le \exp\Big(\frac{N^2 \epsilon^4}{2d} \alpha(\sigma) \Big),
		\end{align*}
		for any cycle decomposition $\sigma \in \Sigma_\mathcal{C}$.
	\end{lemma}
	\begin{proof}		
		We first seek to control the value $b^Tb'$. Using the vectors $a \in \mathcal{A}_\mathcal{C}(\sigma)$, we have that $b^Tb' = \sum_{i \in s} |\sigma_i| a_i a'_i$, where $a$ corresponds to the direction of the cycles for the vector $b$, and $a'$ for $b'$.
		
		Recalling that we chose our $M = 2^{|\sigma|}$ perturbations based on orienting the cycle decomposition $\sigma$, we begin where Lemma~\ref{lemma:bound1} left off and have
		\begin{align*}
		\chi^2(\mathbb{P}_1,\mathbb{P}_0) + 1 
		&\leq \frac{1}{M^2}\sum_{b,b'}\exp\Big(\frac{N \epsilon^2}{d}b^Tb'\Big)\\
		&=\frac{1}{2^{2|\sigma|}}\sum_{a,a'}\exp\Big(\frac{N \epsilon^2}{d}\sum_{\sigma_i \in \sigma} |\sigma_i| a_i a'_i \Big)\\
		&=\mathbb{E}\Big[\prod_{\sigma_i \in \sigma}\exp\Big(\frac{N \epsilon^2}{d} |\sigma_i| a_i a'_i \Big)\Big]\\
		&=\prod_{\sigma_i \in \sigma}\mathbb{E}\Big[\exp\Big(\frac{N \epsilon^2}{d} |\sigma_i| a_i a'_i \Big)\Big]\\
		&=\prod_{\sigma_i \in \sigma}\Big(\frac{1}{2}\exp\Big(\frac{N \epsilon^2}{d}|\sigma_i| \Big) + \frac{1}{2}\exp\Big(\frac{-N \epsilon^2}{d}|\sigma_i| \Big)\Big)\\
		&\leq \prod_{\sigma_i \in \sigma}\Big(\exp\Big(\frac{N^2 \epsilon^4}{2d^2}(|\sigma_i|)^2 \Big)\Big)\\
		&=\exp\Big(\frac{N^2 \epsilon^4}{2d}\alpha(\sigma) \Big),
		\end{align*}
		where the second line follows from controlling $b^Tb'$ with $a$ vectors, the third line follows from the second line being equivalent to the expectation over every pair of vectors $a$, $a'$, the fourth line follows from the fact that each element of vector $a$ is independent from the other elements, the fifth line follows from evaluating the expectation for a pair of independent Rademacher variables, the sixth line follows from the fact that $\frac{1}{2}\exp(x) + \frac{1}{2}\exp(-x) \leq \exp(\frac{x^2}{2})$, and the final line follows from rearranging the terms and applying the definition of $\alpha(\sigma)$. 
	\end{proof}
	\section{Proof of main result}
	\label{sec:thmproof}
	
	We now use Lemma~\ref{lemma:bound1} and \ref{lemma:bound2} to prove the main result in Theorem~\ref{thm:lower_bound}. 
	
	\begin{proof}
		\begin{align*}
		R_{N,\delta}(\Piia) 
		&\geq \frac{1}{2} - \frac{1}{4}||\mathbb{P}_0 - \mathbb{P}_1||_{TV}
		\geq
		\frac{1}{2} - \frac{1}{4}\sqrt{\chi^2(\mathbb{P}_1,\mathbb{P}_0)} \\
		&\geq \frac{1}{2} - \frac{1}{4}\Big(\exp\Big(\frac{N^2 \epsilon^4}{2d}\alpha(\sigma) \Big) - 1\Big)^{\frac{1}{2}}  \\
		&\geq
		\frac{1}{2} - \frac{1}{4}\Big(\exp\Big(\frac{8\mu(\sigma)^4\alpha(\sigma)N^2 \delta^4}{d} \Big) - 1\Big)^{\frac{1}{2}},
		\end{align*}
		where the last step follows from substituting $\epsilon = 2\mu(\sigma)\delta$. The final expression makes clear the dependence of the lower bound on $\sigma$. Namely, the bound  applies to any $\sigma$ of  $G_\mathcal{C}$, and so $\sigma$ can be chosen to produce the best lower bound. This observation means we can restate the risk bound as
		\begin{align*}
		R_{N,\delta}(\Piia) 
		&\geq \max_{\sigma \in \Sigma_{\mathcal{C}}} \frac{1}{2} - \frac{1}{4}\Big(\exp\Big(\frac{8\mu(\sigma)^4\alpha(\sigma)N^2 \delta^4}{d} \Big) - 1\Big)^{\frac{1}{2}},
		\end{align*}
		where we remind the reader that $\Sigma_\mathcal{C}$ is the collection of all possible edge-disjoint cycle decompositions of $G_\mathcal{C}$.
		Indeed, the smaller the size of each cycle in the decomposition of $G_\mathcal{C}$, the stronger the lower bound gets, since both $\mu(\sigma)$, the average cycle size, and $\alpha(\sigma)$, the cycle dispersion index, become smaller. 
	\end{proof}
We state the following risk bound for arbitrary $\mathcal C$ as a corollary, using our best arguments for bounding $\mu(\sigma)^4 \alpha(\sigma)$ for general $\mathcal C$ in terms of $d$ and $n$, though note that for specific $\mathcal C$ better bounds are possible, as discussed further in Section~\ref{sec:specifics}.

	\begin{cor}\label{cor:global_lower_bound}
		For any $G_\mathcal{C}$, $R_{N,\delta}(\Piia)$ may be globally lower bounded as
		\begin{align*}
		R_{N,\delta}(\Piia) \geq \frac{1}{2} - \frac{1}{4}\Big(\exp\Big(
		c \min \Big\{\frac{(d^4\log(n)^5 + d^3(n-\log(n))n\log(n)^4)N^2 \delta^4}{d(d - 4n + 8\log(n))^4}, \frac{n^5N^2 \delta^4}{d} \Big\} 
		\Big) - 1\Big)^{\frac{1}{2}}.
		\end{align*}
	\end{cor}
	\begin{proof}
		Using Lemma \ref{lemma:Gc_cycle_decomposition} in the Appendix, the quantities $\alpha(\sigma)$ and $\mu(\sigma)$ may be globally upper bounded for any $G_\mathcal{C}$ by the following expressions: 
		\begin{align*}
		&\mu(\sigma) \leq \min \Big\{\Big(\frac{d}{d - 4n + 8\log_2 (n)}\Big)4\log_2 (n), 2n\Big\}, \\
		&\alpha(\sigma) \leq \min \Big\{4 \log_2(n) + \frac{4n(2n - 4 \log_2(n))}{d}, 2n\Big\}.
		\end{align*}
		Using these expressions and with a bit of algebra, we can calculate $\mu(\sigma)^4\alpha(\sigma)$ and substitute it into the main result of Theorem \ref{thm:lower_bound} to obtain the lower bound of the corollary.
		
	\end{proof}
	It is worthwhile to briefly unpack the expressions for $\mu(\sigma)$ and $\alpha(\sigma)$ that result in the rather complex lower bound of Corollary \ref{cor:global_lower_bound}. Noteworthy is the $\min$ that appears in both expressions. The term results from the limitations of what can be said about cycle decompositions for general bipartite graphs. On one hand, the pigeonhole principle guarantees that for all comparison incidence graphs $G_\mathcal{C}$, cycles can be of size at most $2n$. Hence for Eulerian graphs we can ensure that the average cycle size $\mu(\sigma)$ and the cycle dispersion index $\alpha(\sigma)$ are each always at most $2n$. 
	
	A cycle size that is linear in $n$ is, however, unsatisfying for denser bipartite graphs, where we would expect cycles to be considerably smaller than $2n$. For the denser bipartite graphs, Lemma \ref{lemma:bipartite_cycle_decomposition} in the Appendix reveals that most cycles are indeed of order $\log(n)$, but a (often small) fraction of cycles will be of order $n$. This statement is tight for bipartite graphs. When $d$ is very large (generally, more than order $n^2$), such as the ``all subsets'' case we discuss in Section \ref{subsection:all}, the fraction of order-$n$ cycles becomes negligible, and both $\alpha(\sigma)$ and $\mu(\sigma)$ can be upper bounded by some constant times $\log(n)$. However, when $d$ is very small, the graph is sparse and a majority of the cycles are order $n$, and the order $n$ terms dominate the expressions. The complex expressions of $\alpha(\sigma)$ and $\mu(\sigma)$ seek to precisely document how the transition from majority length-$n$ cycles to majority length-$log(n)$ cycles happens as the graph gets more dense, in terms of $d$ and $n$. The complexity then carries over to final risk bound, but also renders it useful for a wide range of scenarios.
	
	As a closing remark, and a preview into the next section, we note that the above reduction to a global lower bound, though illustrative in a scenario where $d$ is exponential in the size of $n$, falls short for when $d$ is smaller. Consider instead the $\mathcal{C}$ that consists of all pairwise comparisons, for $n = 6x+1$ for any natural number $x$; following an elementary result of Kirkman \cite{kirkman1847problem,bollobas1986combinatorics}, the graph $G_\mathcal{C}$ corresponding to such a $\mathcal{C}$ is Eulerian and can always be decomposed into cycles of size 6, making both $\alpha(\sigma)$ and $\mu(\sigma)$ constants. We elaborate on these different cases in greater detail in the following section. 
	
	\section{Lower bounds for specific comparison incidence graphs}
	\label{sec:specifics}
	We find it fruitful to leave both $\alpha(\sigma)$ and $\mu(\sigma)$ as a part of the lower bound; it renders the lower bound sharper, but also provides guidance for experimental design. Suppose $\mathcal{C}$ were chosen in a particular experiment, where the experimenter seeks to discover violations of IIA with a limited sample budget $N$. Our lower bound demonstrates the minimum worst case probability of error for any test for violations of IIA given a budget of samples $N$, but is clearly dependent on the choice of $\mathcal{C}$. The presence of a sample budget creates an essential trade off: while the experimenter would like to broaden the scope of her test, attempting to study more choice sets where IIA could be violated, adding sets to $\mathcal{C}$ would increase the minimum error of any tester---perhaps beyond an acceptable threshold. The minimum error can increase in two ways: due to the resulting increase in $d$ from the additional sets, but also due to the changes in $\alpha(\sigma)$ or $\mu(\sigma)$ from the changing structure of the comparison incidence graph $G_\mathcal{C}$. We use this setting---one of seeking to maximize coverage of choice sets given a fixed number of samples and risk threshold---to motivate a closer inspection of our lower bound. 
	
	Before we proceed, we note that while comparing our lower bounds across different $\mathcal{C}$, we omit a discussion of the mild differences in behavior of the separation $\delta$ across various $\mathcal{C}$. One difference, a consequence of Lemma $\ref{lemma:seperation}$, is that $\delta$ can be at most $(2\mu(\sigma))^{-1}$ and therefore loosely depends on both the support size $d$ and the cyclicality of $G_\mathcal{C}$. We omit this discussion, and more generally omit subscripting $\delta$ to not stray from the total variation distance metric's setting-independent property as the maximal gap in probability between any two distributions. Our omission is consistent with other analyses of testing~\cite{paninski2008coincidence,wei2016sharp,valiant2017automatic}, and consistent with the idea that $\delta$ is regarded as a small value kept as close to zero as allowable by the practical constraints of dimension and sample size.

	\subsection{All subsets}
	\label{subsection:all}
	IIA is defined as a property of the complete choice system of $n$ items, and hence, any true test of IIA must encompass all the subsets of an item universe. We thus focus first on the results of our lower bound on this problem instance. We again consider the slightly easier setting amenable to our analysis where we only study (all) subsets of even size. Over the unique even subsets of a choice system, a simple calculation reveals that every item then appears the following number of times,
	\begin{align*}
	\sum_{k=1}^{\floor{n/2}} \binom{n-1}{2k-1} = 2^{n-2},
	\end{align*}
	which is always an even number. Hence, the resulting comparison graph is Eulerian, and our lower bound applies with $d=n2^{n-2}$. Given the sheer size of $d$, we do not lose much in the flavor of our statement by using the upper bounds for $\alpha(\sigma)$ and $\mu(\sigma)$ from Lemma \ref{lemma:Gc_cycle_decomposition}. In fact, since $d$ is exponentially large in $n$, the expressions for both $\alpha(\sigma)$ and $\mu(\sigma)$ can be upper bounded by $c \log(n)$ where $c$ is some constant\footnote{This argument is detailed in Corollary \ref{cor:all_subsets_cycle_decomposition}, a corollary of Lemma~\ref{lemma:Gc_cycle_decomposition}, in the Appendix.}. We then have the following lower bound for the setting of all subsets of even size, and consequently, a lower bound for the setting of all subsets\footnote{Conveniently, $d=n2^{n-1}$ for the case of all subsets (even or odd sizes), and we can use this value in our lower bound as it is within a constant factor of $d=n2^{n-2}$ for the all even subsets case.}:
	\begin{align}
	R_{N,\delta}(\Piia) \geq \frac{1}{2} - \frac{1}{4}\Big(\exp\Big(\frac{c \log(n)^5N^2 \delta^4}{n2^{n-1}} \Big) - 1\Big)^{\frac{1}{2}}.
	\end{align}
	That is, the best possible test for IIA has worst case error arbitrarily close to $1/2$---the error of a random guess---until the number of sample is exponentially large in $n$. 
	
	The result is insightful because it brings to light a special concept: although IIA is simple to represent, in order to test whether a choice system obeys IIA, the test must also consider every possible way a choice system can violate IIA in order to produce a meaningful certificate. Missing even one of those violations results in a test with arbitrarily large error in the worst case where the deviation is represented by that particular violation. Moreover, the number of deviations from IIA are not just numerous, but also close to IIA distributions in a statistical sense, thereby requiring an immensely large number of samples to guarantee a low probability of error.
	
	Because IIA is a property about a combinatorial system, its conceptual simplicity does not aid in distinguishing it from the myriad of combinatorial alternatives. This point is our main takeaway: even the simplest, most storied theoretical concept in discrete choice is practically impossible to test for. 
	
	In light of this pessimistic result, we proceed to take a more defensive, practical viewpoint. In most practical settings, violations of IIA are not sought after in \textit{any} choice set as much as choice sets of meaningful relevance. Often in practice, choice sets do not exceed a modest fixed size (e.g., pairs or triplets). Moreover, there may be sparsity even among the set of small choice sets, as practical constraints prevent the comparison of all items. In these settings, a practitioner does not ever care to test the combinatorially many different violations of IIA, but rather only cares about a restricted set of violations possible on the given choice sets. In other words, a choice system should  only be regarded as deviating from IIA if it deviates within the sets of interest; else, a failure to reject is tantamount to the choice system obeying IIA.
	
	\subsection{All pairs}
	Consider then the setting where $\mathcal{C}$ is restricted to only contain (all possible) pairs. This setting, sometimes called pairwise comparisons, is prevalent in discrete choice and most commonly found in match-up modelling in sports or games. In match-ups, for instance, the relevant violations of IIA are certainly restricted to pairs: sets of three do not have any meaning. Given a collection $\mathcal{C}$ of all possible pairs of a universe of $n$ items, we observe that an item appears $n-1$ times over all of the unique choice sets in the collection, and so $d=n(n-1)$. Hence, when $n$ is odd, $d$ is even, and then $G_\mathcal{C}$ is Eulerian (since all pairs are already of even size). We divide up the odd $n$ into three different cases, $n=6x+1$, $n=6x+3$, and $n=6x+5$ where $x$ is any natural number. As previewed earlier, by a result due to Kirkman \cite{kirkman1847problem}, in the case of $n=6x+1$, $G_\mathcal{C}$ can always be decomposed into edge disjoint cycles of size $6$. The same argument applies to the $n=6x+3$ case. We come about this result by considering first a graph on $n$ nodes $H$ with undirected edges between two nodes compared in a choice set in $\mathcal{C}$. We note that every triangle in $H$ then maps uniquely to a cycle of size $6$ in the bipartite graph $G_\mathcal{C}$, containing both the items in the vertices of the triangle in $H$, as well as the vertices corresponding to the choice sets those comparisons were made in. Since $\mathcal{C}$ contains all pairwise comparisons, $H$ is the clique $K_n$, and Kirkman's result that $K_n$ cliques can be decomposed into edge-disjoint triangles when $n=6x+1$ and $n=6x+3$ is equivalent to $G_\mathcal{C}$ being decomposed into cycles of size 6. As a result, $\alpha(\sigma) = 6$, a term that can be swept into the constant $c$. 
	
	Then, for $n=6x+5$, separate prior work shows that $H$ can almost be decomposed into triangles, but $4$ edges remain \cite{feder2012packing}. Since the nodes of $H$ have even degree, the removal of all the triangles still leaves all the nodes with even degree. Thus, the remaining edges form a cycle, and hence map to a cycle of size $8$ in $G_\mathcal{C}$. Thus, $\alpha(\sigma) = 6 + \frac{16}{d}$, which can again be upper bounded by a constant and swept into $c$. Thus, for the setting of all pairs with $n$ odd, we have the lower bound,
	\begin{align}
	R_{N,\delta}(\Piia) \geq \frac{1}{2} - \frac{1}{4}\Big(\exp\Big(\frac{c N^2 \delta^4}{n(n-1)} \Big) - 1\Big)^{\frac{1}{2}}.
	\end{align}
	As we would expect, our lower bound is considerably more optimistic in this setting. Indeed, the number of samples $N$ only need to be on the same order as the number of items $n$ for our lower bound to fall away. This mild scaling demonstrates the value of restraining the sets of interest to just the pairs. While one cannot say anything about deviations from IIA beyond the pairs when looking only a pairwise data, such a restriction provides great advantages in sample efficiency.
	
	\subsection{The single big cycle}
	We lastly consider a specialized $\mathcal{C}$ consists of only pair comparisons that form a simple ``cycle'' among the $n$ items: e.g. $\{i,j\}, \{j,k\}, \{k,l\}...\{z,i\}$. Here, $d = 2n$. We note that for this setting, the upper bound on $\alpha(\sigma)$ of $2n$ is sharp; that is, the only simple cycle decomposition of $G_\mathcal{C}$ is $G_\mathcal{C}$ itself, and hence $\alpha(\sigma)$ is no smaller than $2n$. We similarly have $\mu(\sigma)$ = $2n$. We then have, for this setting, a lower bound of:
	\begin{align}
	R_{N,\delta}(\Piia) \geq \frac{1}{2} - \frac{1}{4}\Big(\exp\Big(c n^4N^2 \delta^4 \Big) - 1\Big)^{\frac{1}{2}}.
	\end{align}
	That is to say, our lower bound falls away fast, regardless of $n$. Indeed, as written, the lower bound falls away faster for larger $n$; this errant behavior, however, we conclude comes largely from a likely lack of sharpness in characterizing the projection distance for a simple cycle in Lemma \eqref{lemma:seperation}, a possible area of improvement that we will leave for future work. Ignoring this added factor of $n^4$ contributed by $\mu(\sigma)$, we still see a rate falling off independent of $n$ or $d$. 
	
	This lack of a pessimistic lower bound translates to a cause for optimism, and a guideline for experimental design. In settings of highly limited samples, with an experimenter that can choose the $\mathcal{C}$ for which to allocate those samples, our lower bound ``suggests'' choosing a simple cycle for the allocation.
	In a choice system rife with deviations from IIA, most cycles would contain some of these violations. The benefit of allocating samples to only these cycles, as opposed to a greater number of sets such as all pairs, however, is a low rate of test error. That is, our lower bound guarantees that with a very small number of samples, the setting of ``all pairs'' will necessarily have high error for any test. In the setting of the ``single big cycle'', however, our lower bound does not provide such a pessimistic result; should the test reject, an experimenter could be more confident of the accuracy of the rejection, and still achieve the overall goal of detecting deviations from IIA. Should it not, the experimenter would at least know to seek out different sets to test for violations. By restricting the test to a cycle in these low sample scenarios, a researcher trades a result guaranteed to be errant for a more conservative result likely to be veritable, and is therefore no worse off. This insight is one major deliverable from our structure-dependent finite sample analysis. 
	
	\subsection{A fixed $w(C)$?}
	Unlike an observational study, where $\mathcal{C}$ and $w(C)$ are taken as given, it is not unrealistic in experimental design for both to be \textit{chosen}. In this section, we have so far considered the benefit of carefully choosing $\mathcal{C}$. By potentially also choosing $w(C)$, however, an experimental can potentially allocate samples to every choice set in $\mathcal C$. At first glance, it may appear that our lower bounds do not apply for some well-chosen predetermined $w(C)$. Crucial to our bound is the reduction of the general IIA test \eqref{eq:iiatest}, which encompasses all IIA choice systems in the null (and hence, all $w(C)$), to the test of \eqref{eq:test2}, which considers only the uniform choice system in the null. For the uniform choice system, $w(C) = \frac{|C|}{d}$, a weight proportional to the size of each $C$. The testing reduction is valid because $\eqref{eq:iiatest}$ contains all $w(C)$ within its null. For a fixed $w(C)$, specifically, a $w(C) \neq \frac{|C|}{d}$, the validity of the reduction to test \eqref{eq:test2} is no longer immediately obvious. The matter of a fixed $w(C)$ invokes a deeper question: are our lower bounds, and their associated pessimism, a result of $w(C)$ not being fixed?
	
	Fortunately, our lower bounds apply regardless of whether $w(C)$ is chosen or not. Consider first that $w(C)$ has no bearing at all on whether a choice system is IIA. Next, observe that the lower bounds in this paper are ``minimax'', studying only the worst case scenario for every tester $\phi$. Finally, note that all of the $\mathcal{C}$ studied so far form $G_\mathcal{C}$ that are Eulerian, which have exact cycle decompositions. A consequence of this property is that \textit{every} entry within a choice system is part of \textit{some} cycle, and hence can be part of some deviation from IIA. Because we only bound a test's worst case error, any potential direction of exit is a cause for concern as it may contribute to the worst case error. What is, then, the easiest $w(C)$ for testing IIA for the settings of this work? A $w(C)$ that prioritizes every entry of the choice system (and hence, every avenue of IIA departure) equally. This is a $w(C)$ proportional to size of each $C$, the $w(C)$ of \eqref{eq:test2}'s uniform choice system. Were $w(C)$ not entry-wise proportional, then there would be some entry with less than uniform weight, and hence the worst case for that testing problem would now be strictly harder. Because $w(C) = \frac{|C|}{d}$ is the easiest $w(C)$, a problem with a fixed $w(C)$ could only be harder. Thus, our lower bounds carry over to those settings.
	
	\section{Conclusions}
	We develop a structure-dependent finite sample lower bound on the minimax risk of all hypothesis tests for IIA. A first consequence of this bound is rather grave: no general test for IIA can provide a meaningfully small worst case error without a number of samples exponential in the number of items. Our methods, however, provide more than just pessimism: by characterizing the lower bound for a wide range of $\mathcal{C}$, we demonstrate the value of ``researcher priors''. Indeed, a prior that IIA violations lie only in the set of pairs vastly reduces the complexity of our lower bound. By incorporating the combinatorial structure of $\mathcal{C}$ into our testing problem, we gain further insight into the sample efficiency of specific structures. As a case study, we consider the optimism hinted by a rapidly diminishing ``dimension-free'' lower bound for the special case of a cycle, a feature reminiscent of the complexity of property testing for cyclicality.
	
	While the IIA model is well specified, the complexity of a testing problem is dependent on both the complexity of the model class, as well as the complexity of possible departures. Research priors can limit the set of possible departures from IIA, and provide valuable performance improvements. With this perspective, we briefly assess the function of a widely used existing test for IIA, the Hausmann-McFadden test \cite{hausman1984specification}. With this test, the proposal is to split a dataset and study two scenarios---the maximum likelihood IIA distribution with and without the presence of an item. Such a split may be viewed through the lens of two-sample hypothesis testing, where the question concerns whether the two scenarios produce the same distribution or whether they do not. The Hausmann-McFadden test is {\it not} a general test for IIA. Rather, it may be viewed as testing whether a single item $i$ (the item being removed) changes the IIA projection. Viewing the Hausmann-McFadden test in this lens demonstrates its potential merit---the implicit prior imposed by the test does not seem immediately unreasonable, but its value should be verified empirically. Such a view also defines the tests' limitations; there are many ways of violating IIA that preserve total orderings with and without a single item. 
	
	In contrast to the Hausmann-McFadden test, the \textit{universal logit test} \cite{mcfadden1977application}, a likelihood ratio test comparing the MNL model to an arbitrary choice system, \textit{is} a general test for IIA that is asymptotically valid. However, the accuracy of the universal logit test depends on an accurate estimate of a choice system, which requires samples on the same order as $d$. As expected, this sample complexity exceeds the lower bound in this paper for general tests for IIA, which is (up to log factors for dense $\mathcal{C}$) samples on the same order as $\sqrt{d}$. The performance of the universal logit test and the requirement of our lower bound thus form a wide gap between sufficiency and necessity.
	
	Akin to the Hausmann-McFadden test's implicit prior, yet another notion of a prior stems from a valid model for departures from IIA. Such ``model-based'' tests are scant in the field of discrete choice, for the main reason that the models themselves are often uninterpreble, inferentially intractable, or both. Hausman and McFadden originally proposed a model-based test centered around nested logit model in their early work on testing IIA \cite{hausman1984specification}. A test against a nested logit model requires the specification of a candidate nesting, much as the Hausmann-McFadden test requires the specification of an item for removal. This high-dimensional specification requirement can be a tall order. Furthermore, inference for the nested logit model can be very difficult \cite{benson2016relevance}. Recent work on model-based testing proposes the context-dependent utility model (CDM) \cite{seshadri2019raw}, a model capable of modelling departures from IIA in terms of a ``pairwise dependence on irrelevant alternatives'' (and departures from the class of random utility models (RUMs) more generally), while still exhibiting ease of optimization (a convex likelihood), tractable finite-sample uniform convergence guarantees, and parametric efficiency. A good model can serve as an excellent platform for principled testing of departures from IIA, with the obvious limitation that departures in the blindspots of the model will remain untested.
	
	We do not propose any new tests in this work---only lower bounds on the sample complexity of the best possible tests. We hope that our lower bound results, with proof techniques that harness the structure-dependent boundary between IIA and more general choice systems in the finite sample domain, can open the door to new constructive tests that test IIA rigorously and efficiently, perhaps even optimally (fully allowing for the possibility that our lower bounds may not be tight). Our \textit{sublinear} sample complexity lower bound seems to place the IIA testing problem at the intersection of property testing and sublinear algorithms. Much progress has been made at this intersection very recently due to the considerable attention it has received from the statistics, information theory, and theoretical computer science communities. The wide availability of recent results and the immense practical importance thus makes the pursuit of a sublinear IIA testing procedure especially attractive. 
	
	Indeed, discrete choice has no shortage of important questions that are similarly ripe for study, especially with regards to testing properties (sometimes called axioms) of choice data. We briefly discuss the idea of testing regularity, RUM-representability, and strong stochastic transitivity. \textit{Regularity} asks whether the probability of choosing an alternative from a choice set is non-decreasing if the choice set is strictly enlarged. 
	Regularity is perhaps the most studied property of RUMs, and a test is akin to a test of monotonocity for a set function. Another example takes a step beyond regularity, a procedure which asks whether a choice system is {\it RUM-representable}\footnote{As shown by Falmagne, regularity conditions are the first set of conditions that form the larger set of Block-Marschak inequalities \cite{block1959random} that fully characterize RUMs.}. Though first posed as a testing question by Falmagne \cite{falmagne1978representation}, this question continues to elude researchers. Recent work \cite{jagabathula2018limit} provides a complete structure-dependent characterization of the algorithmic difficulty of estimating RUMs, but does not discuss any statistical difficulties, or testing. Another important challenge would be to develop an efficient procedure for testing {\it strong stochastic transitivity (SST)}, which asks whether a choice system can be determined simply by a total ordering over the items. SST remains crucial to a lasting legacy of empirical work \cite{davidson1959experimental,mclaughlin1965stochastic, tversky1969substitutability}; though estimating a model that satisfies SST is well characterized \cite{chatterjee2015matrix,shah2016stochastically}, the testing question remains largely unexamined. We thus hope that our work invites not only a rigorous rethinking of the IIA testing question, but also opens the door more broadly to a rigorous rethinking of testing in discrete choice.
	
	\bibliographystyle{ACM-Reference-Format}
	\bibliography{iia}
	
	\clearpage
	\begin{appendix}
		\section{Appendix}
		\subsection{Equivalent reformulation with level $\alpha$ tests}
		\label{app:alphatest}
		We define a test $\phi$ as a map, $\phi : (x_1,C_1),...,(x_N,C_N) \mapsto \{0,1\}$ from the dataset to a decision to reject the null $H_0$ as before. This time, however, define $\Phi_{N,\alpha} = \lbrace \phi : \ \sup_{p \in \mathcal{P}_0} p^N(\phi(\mathcal{D}_N) = 1) \leq \alpha \rbrace$ as the set of all level $\alpha$ tests. Since the type I error is always controlled, the risk of interest is now only the type II error. Thus, the minimax risk is:
		\begin{align*}
		R_{N,\delta,\alpha}(\Piia) = \inf_{\phi \in \Phi_{N,\alpha}} \sup_{q \in \Mdelta} q^N(\phi(\mathcal{D}_N) = 0). 
		\end{align*}
		The minimax risk $R_{N,\delta,\alpha}(\mathcal{P}_0)$ is lower bounded as \[1 - \alpha - \frac{1}{2} \Big(\exp\Big(
		\frac{8\mu(\sigma)^4\alpha(\sigma)N^2 \delta^4}{d} -1 
		\Big)^{\frac{1}{2}} \leq R_{N,\delta, \alpha}(\mathcal{P}_0), \] where the remaining terms are defined in the same manner as Theorem \ref{thm:lower_bound}. The result follows from applying Theorem \ref{thm:lower_bound} to the sum of Type I and Type II errors as before, and then simply upper bounding the Type I error by $\alpha$. While straightforward, we restate it this way to make clear the consequence of the bound for level $\alpha$ tests: when $\frac{8\mu(\sigma)^4\alpha(\sigma)N^2 \delta^4}{d}$ is small, the power is just $\alpha$, so the test is just some coin with probability $\alpha$ in the worst case.
		\subsection{Lemmas and facts: choice systems and testing}

		\begin{fact}\label{cvx_hull_dense}
			The convex hull of $\Piia$ is dense in $\Pall$.
		\end{fact}
		\begin{proof}
			The proof of the statement is straightforward, using only the idea that $\Piia$ contains distributions that arbitrarily close to the ``corners'' of the $\textit{choice system}$ simplex. The statement then immediately holds, since convex combinations of simplex corners occupy the entire simplex. 
		\end{proof}
		
		\begin{lemma}\label{lemma:iia_invariance}
			Given two choice systems $q_1$ and $q_2$ defined over the same collection $\mathcal{C}$ such that 
			$$
			\sum_{j \in C} {q_1}_{(j,C)} = \sum_{j \in C} {q_2}_{(j,C)} = \mu_C, \ \ 
			\forall C \in \mathcal{C},$$ and 
			$$\sum_{\substack{C \in \mathcal{C}\\ C \ni i}} {q_1}_{(i,C)} = \sum_{\substack{C \in \mathcal{C}\\ C \ni i}} {q_2}_{(i,C)} = \alpha_i, \ \ 
			\forall i \in \mathcal{X},$$
			then, 
			$$
			\argmin_{p \in \Piia} 
			D_{\text{KL}}(q_1||p) = 
			\argmin_{p \in \Piia} 
			D_{\text{KL}}(q_2||p).
			$$
		\end{lemma}
		\begin{proof}
			From the definition of KL divergence and IIA, and removing terms that don't affect the argument, we have
			\begin{align*}
			\text{arg} \inf_{p \in \Piia} D_{KL}(q||p) 
			&= \text{arg} \inf_{p \in \Piia} \sum_{C \in \mathcal{C}} \sum_{j \in C} q_{(j,C)} \log\Big(\frac{q_{(j,C)}}{p_{(j,C)}}\Big)\\
			&= \text{arg} \inf_{p \in \Piia} \sum_{C \in \mathcal{C}} \sum_{j \in C} q_{(j,C)} \log(q_{(j,C)}) - \sum_{C \in \mathcal{C}} \sum_{j \in C} q_{(j,C)} \log(p_{(j,C)})\\
			&=\text{arg} \inf_{w \in \Delta_m, \theta \in \mathbb{R}^n}  
			- \sum_{C \in \mathcal{C}} \sum_{j \in C} q_{(j,C)} \log\Big(\frac{w_C \exp(\theta_j)}{\sum_{k \in C} \exp(\theta_k)}\Big).
			\end{align*}
			We can further rewrite this expression using properties of logarithms and rearranging terms:
			\begin{align*}
			\text{arg} \inf_{p \in \Piia} D_{KL}(q||p)
			&=\text{arg} \inf_{w \in \Delta_m, \theta \in \mathbb{R}^n} 
			- \sum_{C \in \mathcal{C}} \log(w_C)\sum_{j \in C} q_{(j,C)} - \sum_{C \in \mathcal{C}} \sum_{j \in C} q_{(j,C)} \log\Big(\frac{\exp(\theta_j)}{\sum_{k \in C} \exp(\theta_k)}\Big) \\
			&=\text{arg} \inf_{w \in \Delta_m, \theta \in \mathbb{R}^n} 
			- \sum_{C \in \mathcal{C}} \sum_{k \in C} q_{(k,C)}\log(w_C) - \sum_{C \in \mathcal{C}} \sum_{j \in C} q_{(j,C)} \log\Big(\frac{\exp(\theta_j)}{\sum_{k \in C} \exp(\theta_k)}\Big).
			\end{align*}
			Next we notice that we can optimize over $w$ by applying Fact \ref{max_entropy} (see below), that is, $w^\star = \sum_{k \in C} q_{(k,C)}$, and a few additional steps of algebra:
			\begin{align*}
			\text{arg} \inf_{p \in \Piia} D_{KL}(q||p)
			&=\text{arg} \inf_{\theta \in \mathbb{R}^n} 
			- \sum_{C \in \mathcal{C}} \sum_{j \in C} q_{(j,C)} \log\Big(\frac{\exp(\theta_j)}{\sum_{k \in C} \exp(\theta_k)}\Big)\\
			&= \text{arg} \inf_{\theta \in \mathbb{R}^n} - \sum_{C \in \mathcal{C}} \sum_{j \in C} \theta_j q_{(j,C)} + \sum_{C \in \mathcal{C}} \sum_{j \in C} q_{(j,C)} \log(\sum_{k \in C} \exp(\theta_k))\\
			&= \text{arg} \inf_{\theta \in \mathbb{R}^n} - \sum_{j \in \mathcal{X}} \theta_j \sum_{\substack{C \in \mathcal{C}\\ C \ni j}} q_{(j,C)} + \sum_{C \in \mathcal{C}} \log(\sum_{k \in C} \exp(\theta_k)) \sum_{j \in C} q_{(j,C)}.
			\end{align*}
			
			Now, we see that, 
			\begin{align*}
			\text{arg} \inf_{p \in \Piia} D_{KL}(q_1||p) 
			&= \text{arg} \inf_{\theta \in \mathbb{R}^n} - \sum_{j \in \mathcal{X}} \theta_j \alpha_j + \sum_{C \in \mathcal{C}} \mu_C \log(\sum_{k \in C} \exp(\theta_k)) \\
			&= \text{arg} \inf_{p \in \Piia} D_{KL}(q_2||p).
			\end{align*}
			This result concludes the proof.
		\end{proof}
		
			\begin{fact}\label{max_entropy}
			For any $q \in \Delta_s$, 
			\begin{align}
			\inf_{p \in \Delta_s}
			-\sum_{i=1}^s q_i \log(p_i) 
			=  H(q),
			\end{align}
			where $H(q)$ denotes the entropy of $q$.
		\end{fact}
		\begin{proof}
			\begin{align*}
			\inf_{p \in \Delta_s} -\sum_{i=1}^s q_i \log(p_i) &= \inf_{p \in \Delta_s} -\sum_{i=1}^s q_i \log(p_i) + \sum_{i=1}^s q_i \log(q_i) - \sum_{i=1}^s q_i \log(q_i)\\ 
			&= \inf_{p \in \Delta_s} \sum_{i=1}^s q_i \log\Big(\frac{q_i}{p_i}\Big)+ \sum_{i=1}^s q_i \log\Big(\frac{1}{q_i}\Big)\\
			&= \inf_{p \in \Delta_s} D_{KL}(q||p) + H(q)\\
			&=H(q),
			\end{align*}
			where the last line follows from the non-negativity of KL divergence, and it being zero if and only if $p = q$. 
		\end{proof}
		
\subsection{Lemmas and facts: cycle structure}
		
		\begin{lemma}\label{lemma:general_cycle_decomposition}
			Given an arbitrary graph $G$ with $n$ vertices and $d$ edges, there exists a cycle decomposition of the edge set into cycles of size at most $2\log_2(n)$ with at most $2n$ edges remaining.
		\end{lemma}
		\begin{proof}
			See \cite{chu2018graph}, Theorem 3.10 and its proof in Section 8 therein.
		\end{proof}
		\begin{lemma}\label{lemma:bipartite_cycle_decomposition}
			Given a bipartite graph $G = (V_1, V_2, E)$ with $n_1 = |V_1|$ and $n_2 = |V_2|$ vertices in each part (wlog $n_2 \geq n_1$) and $d = |E|$ edges, there exists a cycle decomposition of the edge set into cycles of size at most $2 \floor{2 \log_2(n_1)}$ with at most $\min\{2n_1 + n_2, 4n_1 + n_2^\text{(odd)}\}$ edges remaining, where $n_2^\text{(odd)} \leq n_2$ represents the number of vertices of odd degree on the second vertex set.
		\end{lemma}
		\begin{proof}
			The proof follows the same general structure as that of Lemma \ref{lemma:general_cycle_decomposition}, but modifies a few steps to directly incorporate properties of the bipartite graph, and therefore produces results more specific to bipartite graphs. We begin by stating a procedure to construct a cycle decomposition on an arbitrary bipartite graph $G$ with some edges remaining. 
			
			First, repeatedly remove vertices, along with their incident edges, of degree 2 or less from $V_1$ and of degree 1 or less from $V_2$. Second, perform a breadth-first search (BFS) starting from any remaining vertex in $V_1$ until a cycle $\sigma_1$ is found. Add the cycle to the collection of cycles $\sigma$, and remove the edges from $G$. Repeat these two steps, adding cycles $\sigma_i$ to $\sigma$ until no edges remain in $G$. 
			
			Now, we can analyze this procedure. Note that the only edges removed from $G$ that are also not added to $\sigma$ are from vertices of degree 2 or less from $V_1$ and from vertices of degree 1 or less from $V_2$. This fact allows us to upper bound the number of edges not contained in the cycle decomposition by $2n_1 + n_2$. 
			
			Now, consider any instance of the second step of the procedure above. Since all vertices of degree less than 2 have been removed from $V_1$, $V_1$ has vertices of minimum degree 3. Similarly, since all vertices of degree less than 1 have been removed from $V_2$, $V_2$ only has vertices of minimum degree 2. Thus, the arbitrary vertex in $V_1$ at the start of the BFS is connected to at least three vertices in $V_2$. Those three verticies, are in turn connected to at least one new vertex each in $V_1$ if a cycle is not found, all of which are in turn connected to two new vertices each in $V_2$ if a cycle is not found, and so on. Thus, as long as we don't find a cycle, the BFS tree grows at least as a complete binary tree does every time we go to $V_2$. Since there are no leaf nodes in this graph (no vertex has degree 1), a cycle must eventually be found. The maximum depth of the BFS tree is then at most $\floor{2 \log_2(n_1)}$, twice the depth of a complete binary tree since that is the maximum depth before all nodes in $V_1$ are visited and a cycle has to be found. The cycle can then be at most twice the depth of the BFS tree, which upper bounds the size of the cycle as $2\floor{2 \log_2(n_1)}$. Thus, the key difference in this step of the proof when contrasted with that of Lemma \ref{lemma:general_cycle_decomposition} is using $V_2$ as a bridge back to $V_1$ and using $V_1$ only to exponentially grow the BFS tree. The analytical segmentation of the vertices then allows a bound on the size of the cycle explicitly in terms of the number of vertices of $V_1$, which is useful when the vertices are highly imbalanced in the two parts.
			
			Since both steps can always finish as long as $G$ has edges and the second always strictly reduces the number of edges of $G$, the algorithm eventually terminates and produces a cycle decomposition with cycles of size at most $2\floor{2 \log_2(n_1)}$ with at most $2n_1 + n_2$ edges that remain. 
			
			We close by furnishing the proof for the slightly sharper bound on the edges that remain that is stated in the Lemma. Consider first that removing the edges corresponding to a cycle in a graph modifies the degree of every vertex by an even number, and hence does not change the parity of the vertex's degree. Next, consider that during the first step, edges are removed due to the removal of vertices in $V_2$ only if those vertices are degree 1. Thus, vertices in $V_2$ must either begin with odd parity in order for one of their edges to not contribute to the cycle decomposition, or they must be changed to odd parity due to edges removed from the removal of vertices from $V_1$ in the first step. Since vertices removed from $V_1$ in the first step can have degree at most 2, their removal changes the parity of at most $2n_1$ vertices in $V_2$. This argument yields an alternate upper bound of edges removed that do not contribute to the cycle decomposition, $2n_1 + 2n_1 + n_2^\text{(odd)}$, and concludes the proof. 
		\end{proof}
		
		\begin{lemma}\label{lemma:Gc_cycle_decomposition}
			Any Eulerian comparison incidence graph $G_\mathcal{C}$, defined in Section \ref{subsection:comparison_graph}, has a cycle decomposition $\sigma$ where all but at most $\min(2n+m, 4n)$ edges (i.e., at most $\max(d-\min(2n+m, 4n), 0)$ edges) contribute to cycles of size at most $2\floor{2 \log_2(n)}$ and the rest (at most $\min(2n+m, 4n)$) contribute to cycles of size at most $2n$. We then have the following bounds on $\mu(\sigma)$ and $\alpha(\sigma)$: 
			\begin{align*}
			\mu(\sigma) \leq \min \Big\{\Big(\frac{d}{d - 4n + 8\log_2 (n)}\Big)4\log_2 (n), 2n\Big\} \\
			\alpha(\sigma) \leq \min \Big\{4 \log_2(n) + \frac{4n(2n - 4 \log_2(n))}{d}, 2n\Big\}.
			\end{align*}
		\end{lemma}
		\begin{proof}
			Apply Lemma \ref{lemma:bipartite_cycle_decomposition} to $G_\mathcal{C}$, with $n_1 = n$ and $n_2 = m$. Since $G_\mathcal{C}$ is Eulerian, $n_2^\text{(odd)} = 0$, and so we know that $G_C$ has a cycle decomposition with cycles of size at most $2 \floor{2 \log(n_1)}$ with at most $\min(2n + m, 4n)$ edges that remain. Again since $G_\mathcal{C}$ is Eulerian, and removing cycles from $G_\mathcal{C}$ does not change this fact, the remaining $\min(2n + m, 4n)$ edges can always be decomposed into cycles of size at most $2n$. The latter bound on the cycle size follows from the application of the pigeonhole principle described in Section $\ref{sec:relax}$. When combined, these steps form the $sigma$ that possesses the properties stated in the Lemma.
			
			Naturally, the short cycles from Lemma \ref{lemma:bipartite_cycle_decomposition} are only guaranteed if $d$ exceeds $\min(2n + m, 4n)$, else we use the extremal bound on cycle size of $2n$ for all edges in $d$. Note that $2\floor{2 \log_2(n)} \leq 2n$ for all $n \in \mathbb{N}$, so the cycles from Lemma \ref{lemma:bipartite_cycle_decomposition} are always smaller than those guaranteed by the extremal bound. We then have the following lower bound on $|\sigma|$, the number of cycles in the cycle decomposition:
			\begin{align*}
			|\sigma| \geq \max \Big\{\frac{d - 2n - \min(2n, m)}{2\floor{2 \log_2(n)}} + \frac{2n + \min(2n, m)}{2n}, \frac{d}{2n}\Big\}.
			\end{align*}
			We replace $\min(2n,m)$ with $2n$ to weaken the above bound at the benefit of making the expression much simpler, especially for its use bounding $\mu(\sigma)$ and $\alpha(\sigma)$. We also weaken $2\floor{2 \log_2(n)}$ to $4 \log_2(n)$ for the same purpose. For the sharpest results, we invite the reader to carry the original bound over to $\mu(\sigma)$ and $\alpha(\sigma)$ in a manner similar to what we do below. We then have:
			\begin{align*}
			|\sigma| \geq \max \Big\{\frac{d - 4n}{4\log_2 (n)} + 2, \frac{d}{2n}\Big\}.
			\end{align*}
			Since $\mu(\sigma) = \frac{d}{|\sigma|}$, the preceding bound immediately yields an upper bound on $\mu(\sigma)$:
			\begin{align*}
			\mu(\sigma) \leq \frac{d}{\max \Big\{\frac{d - 4n}{4\log_2(n)} + 4, \frac{d}{2n}\Big\}} = \min \Big\{\Big(\frac{d}{d - 4n + 8\log_2 (n)}\Big)4\log_2 (n), 2n\Big\}.
			\end{align*}
			Now, we bound $\alpha(\sigma) = \frac{1}{d}\sum_{\sigma_i \in \sigma} |\sigma_i|^2$. Clearly, $\alpha(\sigma)$ grows with the size of the cycles in $\sigma$, so we may upper bound it by considering the largest size the cycles in $\sigma$ could be. Recall that $2\floor{2 \log_2(n)} \leq 2n$ for all $n \in \mathbb{N}$, so the cycles from Lemma \ref{lemma:bipartite_cycle_decomposition} are always smaller than those guaranteed by the extremal bound of $2n$. Thus, the more size-$2n$ cycles in our cycle decomposition, the higher $\alpha(\sigma)$. Since we can have at most $4n$ edges contribute to cycles of size at most $2n$, if $d \leq 4n$, then we have:
			\begin{align*}
			\alpha(\sigma) \leq \frac{1}{d} \frac{d}{2n} (2n)^2 = 2n.
			\end{align*}
			Otherwise, we have:
			\begin{align*}
			\alpha(\sigma) \leq \frac{1}{d} \Big(\frac{d - 4n}{2\floor{2 \log_2(n)}}(2\floor{2 \log_2(n)})^2 + \frac{4n}{2n}(2n)^2 \Big) \leq 4 \log_2(n) + \frac{4n(2n - 4 \log_2(n))}{d},
			\end{align*}
			where the first inequality follows from computing the formula with all $4n$ edges in cycles at their maximum size of $2n$ and the remaining edges in cycles at their maximum size of $2\floor{2 \log_2(n)}$. The second inequality comes from rearranging terms and weakening $2\floor{2 \log_2(n)}$ to $4 \log_2(n)$. Putting the two cases together, we have:
			\begin{align*}
			\alpha(\sigma) \leq \min \Big\{4 \log_2(n) + \frac{4n(2n - 4 \log_2(n))}{d}, 2n\Big\}.
			\end{align*}
		\end{proof}
		\begin{cor}[]\label{cor:all_subsets_cycle_decomposition}
			For the comparison graph of all subsets of even size described in Section \ref{subsection:all}, $\mu(\sigma)$ and $\alpha(\sigma)$ are bounded as follows: 
			\begin{align*}
			\mu(\sigma) &\leq 5\log_2 (n), \\
			\alpha(\sigma) &\leq 5\log_2 (n).
			\end{align*}
		\end{cor}
		\begin{proof}
			Applying Lemma \ref{lemma:Gc_cycle_decomposition} with $d = n2^{n-2}$ results in the following expressions:
			\begin{align*}
			\mu(\sigma) \leq \min \Big\{\Big(\frac{n2^{n-2}}{n2^{n-2} - 4n + 8\log_2 (n)}\Big)4\log_2 (n), 2n\Big\} \\
			\alpha(\sigma) \leq \min \Big\{4 \log_2(n) + \frac{4n(2n - 4 \log_2(n))}{n2^{n-2}}, 2n\Big\}.
			\end{align*}
			Clearly, as $n$ grows, the exponential growth of $d$ results in both bounds quickly converging to $4 \log_2(n)$. For small values of $n$, a simple plot immediately reveals that for $n \geq 2$, the expressions never exceed $5 \log_2(n)$. 
		\end{proof}
	\end{appendix}
	
\end{document}